\documentclass[10pt]{amsart}

\usepackage{color}
\usepackage{graphicx, tikz}
\usepackage{xcolor}
\usepackage{cancel}
\usepackage{amssymb,amsfonts,amsthm,amsmath,calligra, enumerate, stmaryrd}
\usepackage{extarrows}
\usepackage[utf8]{inputenc}
\usepackage{accents}
\usepackage{slashed}
\usepackage{yfonts}
\usepackage{mathrsfs,pifont}
\usepackage{mathtools}

\usepackage[bookmarksnumbered=true, bookmarksopen=true]{hyperref}
\usepackage[margin=1in]{geometry}
\usepackage[all]{xy}

\theoremstyle{definition}

\newtheorem{Theorem}{Theorem}[section]
\newtheorem{Proposition}[Theorem]{Proposition}
\newtheorem{Lemma}[Theorem]{Lemma}

\newtheorem{Definition}[Theorem]{Definition}

\newtheorem{Remark}[Theorem]{Remark}
\newtheorem{Example}[Theorem]{Example}

\newtheorem{Setting}[Theorem]{Setting}

\numberwithin{equation}{section}

\def\ben{\begin{eqnarray*}}
\def\een{\end{eqnarray*}}

\newcommand{\calH}{\mathcal{H}}

\newcommand{\bA}{\mathbb{A}}

\newcommand{\C}{\mathbb{C}}
\newcommand{\GL}{\mathrm{GL}}

\newcommand{\bC}{\mathbb{C}}
\newcommand{\bD}{\mathbb{D}}

\newcommand{\bN}{\mathbb{N}}

\newcommand{\bQ}{\mathbb{Q}}
\newcommand{\bR}{\mathbb{R}}

\newcommand{\bZ}{\mathbb{Z}}

\newcommand{\bv}{\mathbf{v}}

\newcommand{\bd}{\mathbf{d}}

\newcommand{\cA}{\mathcal{A}}

\newcommand{\cC}{\mathcal{C}}

\newcommand{\cH}{\mathcal{H}}

\newcommand{\cM}{\mathcal{M}}
\newcommand{\cN}{\mathcal{N}}

\newcommand{\fC}{\mathfrak{C}}

\newcommand{\sA}{\mathsf{A}}

\newcommand{\sS}{\mathsf{S}}
\newcommand{\sT}{\mathsf{T}}

\newcommand{\sF}{\mathsf{F}}

\newcommand{\sw}{\mathsf{w}}

\newcommand{\add}{\mathrm{add}}
\newcommand{\rem}{\mathrm{rem}}
\newcommand{\aux}{\mathrm{aux}}
\newcommand{\leg}{\mathrm{leg}}

\newcommand{\Attr}{\operatorname{Attr}}
\newcommand{\Aut}{\operatorname{Aut}}

\newcommand{\can}{\operatorname{can}}

\newcommand{\diag}{\operatorname{diag}}

\newcommand{\Fix}{\operatorname{Fix}}

\newcommand{\Hom}{\operatorname{Hom}}
\newcommand{\Ext}{\operatorname{Ext}}

\newcommand{\Lie}{\operatorname{Lie}}

\newcommand{\Res}{\operatorname{Res}}
\newcommand{\rk}{\operatorname{rk}}

\newcommand{\Spec}{\operatorname{Spec}}
\newcommand{\Stab}{\operatorname{Stab}}

\newcommand{\D}{\operatorname{D}}

\newcommand{\IC}{\operatorname{IC}}
\newcommand{\id}{\operatorname{id}}
\newcommand{\Perv}{\operatorname{Perv}}

\newcommand{\In}{\operatorname{in}}
\newcommand{\Out}{\operatorname{out}}

\makeatletter
\newcommand{\ostar}{\mathbin{\mathpalette\make@circled\star}}
\newcommand{\make@circled}[2]{%
  \ooalign{$\m@th#1\smallbigcirc{#1}$\cr\hidewidth$\m@th#1#2$\hidewidth\cr}%
}
\newcommand{\smallbigcirc}[1]{%
  \vcenter{\hbox{\scalebox{0.77778}{$\m@th#1\bigcirc$}}}%
}
\makeatother

\newcommand*\circled[1]{\tikz[baseline=(char.base)]{
    \node[shape=circle,draw] (char) {#1};}}
\newcommand*\fillcircled[1]{\tikz[baseline=(char.base)]{
    \node (char) {#1};
    \filldraw[color=black, fill=white] (char) circle (.35);
    \node at (char) {#1};}}

\title{Symmetric quiver varieties and critical stable envelopes}

\author{Yalong Cao}
\address{Morningside Center of Mathematics, Institute of Mathematics \& State Key Laboratory of Mathematical Sciences, Academy of Mathematics and Systems Sciences, Chinese Academy of Sciences, Beijing, China}
\email{yalongcao@amss.ac.cn}
\author{Andrei Okounkov} 
\address{Department of Mathematics, Columbia University, New York, U.S.A.}
\email{okounkov@math.columbia.edu} 
\author{Yehao Zhou}
\address{\parbox{\linewidth}{Center for Mathematics and Interdisciplinary Sciences, Fudan University, Shanghai 200433, China\\
Shanghai Institute for Mathematics and Interdisciplinary Sciences (SIMIS), Shanghai 200433, China}}
\email{yehao.zhou@simis.cn}
\author{Zijun Zhou}
\address{School of Mathematical Sciences, Shanghai Jiao Tong University, Shanghai, China}
\email{zijun.zhou@sjtu.edu.cn}


\begin{document}

\begin{abstract}
Symmetric quiver varieties with potentials are natural generalizations of 
Nakajima quiver varieties, and their equivariant critical cohomologies provide more flexible settings for geometric representation theory and enumerative geometry. In this paper, we study their geometric properties and show that they behave like universally deformed Nakajima quiver varieties. Based on this, we provide a new proof of the existence of critical stable envelopes on them. Following an idea of Nakajima, we give a sheaf theoretic interpretation of critical stable envelopes by 
the hyperbolic restriction in the affinization of symmetric quiver varieties. The associativity of hyperbolic restrictions implies the triangle lemma of critical stable envelopes. 
\end{abstract}

\maketitle

\section{Introduction}

Stable envelopes are introduced in the cohomology of Nakajima quiver varieties by Maulik and Okounkov \cite{MO}. 
They have important applications to geometric representation theory and enumerative geometry. 
In \cite{COZZ1}, the authors start to develop the theory of stable envelopes on critical loci, both in critical cohomology and critical $K$-theory. 
Their applications to the study of shifted quantum groups and curve counting on critical loci are explored in \cite{COZZ2, COZZ3}. 
More specifically, our critical loci are defined out of a smooth variety $X$ with a torus action by $\sT$, and a $\sT$-invariant regular function 
$$\sw\colon X\to \C. $$
Analogues to Nakajima varieties \cite{Nak01, Nak02}, for concreteness, we consider here the case when $X$ is a moduli space of quiver representations (without relations).
We further require that $X$ to be \textit{symmetric}, in the sense that the quiver as well as the framings are symmetric 
(see Definition \ref{def of sym quiver}). In this case, we call $X$ to be a \textit{symmetric quiver variety}. 

The readers shall not confuse quiver varieties here with Nakajima quiver varieties, where the former are defined without any relations on the quiver, while the latter 
are zero loci of certain moment maps. In fact, a Nakajima quiver variety can be realized as the critical locus of a symmetric quiver variety with the canonical cubic potential function, and the corresponding equivariant Borel-Moore homology and $K$-theory is identified with the equivariant critical cohomology and critical $K$-theory
e.g.~\cite[Ex.~6.13, Rmk.~6.14]{COZZ1}. In this sense, symmetric quiver varieties with potentials  provide more flexible settings for 
studying problems arising from geometric representation theory and enumerative geometry. 

In order to study problems in these more general settings, it is important to understand the geometry  of symmetric quiver varieties.
The first theorem of this paper is to show that symmetric quiver varieties behave like universally deformed Nakajima varieties, 
which processes flat morphisms to affine spaces, and having a list of remarkable geometric properties.  
More precisely, we have: 
\begin{Theorem}\label{main thm}
Let $X$ be a symmetric quiver variety.
\begin{enumerate}
\item There is a commutative diagram 
\begin{equation} 
\xymatrix{
\widetilde{X} \ar[r]^{\widetilde{\tau}} \ar[d]_{\pi_X} & \mathcal{H} \ar[d]^{p}\\
X \ar[r]^{\tau\,\,\,} & \mathcal{H}/\Gamma\,,
}
\nonumber \end{equation}
where
\begin{itemize}
    \item $\widetilde{X}$ is the product of a vector space and a universally deformed Nakajima variety,
        \item $\pi_X$ is proper and generically finite,
    \item $\widetilde{\tau}$ is a smooth morphism to an affine space $\calH$ such that for any $h\in \calH$ the fiber $\widetilde{\tau}^{-1}(h)$ is isomorphic to a Nakajima quiver variety,
    \item $\Gamma$ is a product of symmetric groups, i.e. $\Gamma\cong \prod_i\mathfrak{S}_{n_i}$, which acts on $\calH\cong \prod_i\bC^{n_i}$ naturally, and $p$ is the quotient map (in particular, $\calH/\Gamma\cong \prod_iS^{n_i}\bC$ is an affine space),
    \item $\tau$ is a flat morphism such that for any $h\in \calH$ the fiber $\tau^{-1}(p(h))$ is an irreducible normal variety with symplectic singularities, and $\pi_X\colon \widetilde{\tau}^{-1}(h)\to \tau^{-1}(p(h))$ is a symplectic resolution.
\end{itemize}
\item There is a $\Gamma$-invariant affine open subset $\calH^\circ\subset \calH$ (and denote $\widetilde{X}^\circ:=\widetilde{\tau}^{-1}(\mathcal H^\circ),X^\circ:=\tau^{-1}(\mathcal H^\circ/\Gamma)$) such that $\tau\colon X^\circ\to \calH^\circ/\Gamma$ is smooth and affine, and the following induced commutative diagram is Cartesian.
\begin{equation*}
\xymatrix{
\widetilde{X}^\circ \ar[r]^{\widetilde{\tau}} \ar[d]_{\pi_X} \ar@{}[dr]|{\Box} & \mathcal{H}^\circ \ar[d]^{p}\\
X^\circ \ar[r]^{\tau\,\,\,} & \mathcal{H}^\circ/\Gamma
}
\end{equation*}
\item The affinization $\Spec\Gamma(X,\mathcal O_X)$ is isomorphic to $\cM_0(Q,\bv,\bd)$, and the affinization morphism $X\to \cM_0(Q,\bv,\bd)$ is proper and birational. Moreover, all irreducible components of $X\times_{\cM_0(Q,\bv,\bd)} X$ except the diagonal $\Delta_X$ have dimension strictly smaller than $\dim X$. 
\end{enumerate}
\end{Theorem} 
In \S \ref{sect on quiv}, we recall the definition of symmetric quiver varieties and their flavour group actions. The proof of the above theorem occupies \S \ref{sect on proof}. 
We remark that Theorem \ref{main thm} has many applications, for example, it is indispensable in the computations of quantum multiplication by divisors on symmetric quiver varieties with potentials studied in \cite{COZZ2}. 
It can also be used to prove the following existence result\footnote{Another proof using different idea is given in \cite{COZZ1}.}.
\begin{Theorem}(Theorem \ref{thm stab corr}, Remark \ref{rmk on stab corr})
Let $X$ be a symmetric quiver variety with a torus $\sT$-action, $\sA\subseteq \sT$ be a self-dual subtorus, and $\fC$ be a chamber of the $\sA$-action. 
For any $\sA$-fixed locus $F$, 
let $\overline{\Attr}_\fC(\Delta_F)$ denote the closure of the attracting set of the diagonal $\Delta_F\subset F\times F$ in $X\times F$. Then 
$\sum_{F\in \Fix_\sA(X)}[\overline{\Attr}_\fC(\Delta_F)]$ is a \textit{stable envelope correspondence}.
In particular, for any $\sT$-invariant potential function $\sw$, we have the induced
\textit{critical stable envelope}:
$$\Stab_{\fC }\colon H^\sT(X^\sA,\sw)\to  H^\sT(X,\sw). $$
\end{Theorem}

Motivated by the work of Nakajima \cite{Nak1}, we give a sheaf theoretic interpretation of critical stable envelopes in terms of 
Braden's hyperbolic restriction functor \eqref{hy loc fun} on the affinization of symmetric quiver varieties \S \ref{sect on conn to stab}, \S \ref{sect on per on sqv}. 
As a consequence of the associativity of the hyperbolic restriction functors (Proposition \ref{prop hyp res associative}), 
the \textit{triangle lemma} of critical stable envelopes follows, i.e., we have\,: 
\begin{Theorem}(Theorem \ref{thm tri lem})
Let $\fC'$ be a face of $\fC$ and $\sA'\subset \sA$ be the subtorus whose Lie algebra is spanned by $\fC'$. 
Then the following diagram 
\begin{equation*}
\xymatrix{
H^\sT(X^\sA,\sw) \ar[rr]^{\Stab_{\fC}} \ar[dr]_{\Stab_{\fC/\fC'}} & & H^\sT(X,\sw), \\
 & H^\sT(X^{\sA'},\sw) \ar[ur]_-{\Stab_{\fC'}} &
}
\end{equation*}
is commutative.
\end{Theorem}

\subsection*{Acknowledgments}
A.O. would like to thank SIMIS for hospitality. We would like to thank Kavli IPMU for bringing us together.

\section{Definitions of symmetric quiver varieties}\label{sect on quiv}


A \textit{quiver} $Q$ is a pair of finite sets $Q=(Q_0,Q_1)$ together with two maps $h,t\colon Q_1\to Q_0$. We will call $Q_0$ the set of nodes and $Q_1$ the set of arrows and $h$ (resp.~$t$) sends an arrow to its head (resp.~tail). If $a\in Q_1$, then we will write $t(a)\to h(a)$ to denote the arrow $a$. We define the 
\textit{adjacency matrix}
\begin{align*}
    (\mathsf Q_{ij})_{i,j\in Q_0}=\#(i\to j)
\end{align*}
and \textit{Cartan matrix}
\begin{align*}
    (\mathsf C_{ij})_{i,j\in Q_0}=2\delta_{ij}-\mathsf Q_{ij}-\mathsf Q_{ji}.
\end{align*}
Take dimension vectors $\mathbf{v}\in \bN^{Q_0}$,
$\underline{\mathbf d}=(\mathbf{d}_{\mathrm{in}},\mathbf d_{\mathrm{out}})\in \bN^{Q_0}\times \bN^{Q_0}$,  the \textit{space of framed representations} of $Q$ with \textit{gauge dimension} $\mathbf{v}$ and \textit{in-coming framing dimension} $\mathbf{d}_{\mathrm{in}}$ and \textit{out-going framing dimension} $\mathbf d_{\mathrm{out}}$ is
\begin{align}\label{equ rvab}
    R(\mathbf{v},\underline{\bd}):=R(Q,\mathbf{v},\underline{\bd})=\bigoplus_{a\in Q_1}\underbrace{\Hom(\bC^{\mathbf v_{t(a)}},\bC^{\mathbf v_{h(a)}})}_{X_a}\oplus \bigoplus_{i\in Q_0}\left(\underbrace{\Hom(\bC^{\bd_{\mathrm{in},i}},\bC^{\mathbf v_i})}_{A_i}\oplus \underbrace{\Hom(\bC^{\mathbf v_i},\bC^{\bd_{\Out,i}})}_{B_i}\right).
\end{align}
The \textit{gauge group} $G=\prod_{i\in Q_0}\GL(\mathbf v_i)$ naturally acts on $R(Q,\mathbf{v},\underline{\bd})$ by compositions with maps. 

Choose a stability condition $\theta\in \bQ^{Q_0}$ such that $\theta$-semistable representations are $\theta$-stable: 
$$R(Q,\mathbf{v},\underline{\bd})^{ss}=R(Q,\mathbf{v},\underline{\bd})^s\neq \emptyset, $$
and define the \textit{quiver variety} as the GIT quotient:
$$\mathcal M_{\theta}(\mathbf v,\underline{\bd}):=\mathcal M_{\theta}(Q,\mathbf v,\underline{\bd}):=R(Q,\mathbf{v},\underline{\bd})/\!\!/_{\theta} G=R(Q,\mathbf{v},\underline{\bd})^s/G. $$
As we do not impose any relation on the quiver, the above space is a smooth quasi-projective variety. 

We define an action of
\begin{align}\label{edge group}
    G_{\mathrm{edge}}:=\prod_{i,j\in Q_0}\GL(\mathsf Q_{ij})
\end{align}
on $R(Q,\mathbf{v},\underline{\bd})$: given a pair of nodes $i,j\in Q_0$, the contribution of the edges from $i$ to $j$ is $\Hom(\bC^{\mathbf v_i},\bC^{\mathbf{v}_j})\otimes \bC^{\mathsf Q_{ij}}$, then the factor $\GL(\mathsf Q_{ij})$ naturally acts on the second component.

We also consider the following group actions on $R(Q,\mathbf{v},\underline{\bd})$: 
\begin{align}\label{framing group}
    G_{\mathrm{fram}}^{\mathrm{in}}=\prod_{i\in Q_0}\GL(\bd_{\In,i})\curvearrowright \bigoplus_{i\in Q_0}\Hom(\bC^{\bd_{\In,i}},\bC^{\mathbf v_i}),\quad G_{\mathrm{fram}}^{\mathrm{out}}=\prod_{i\in Q_0}\GL(\bd_{\Out,i})\curvearrowright \bigoplus_{i\in Q_0}\Hom(\bC^{\mathbf v_i},\bC^{\bd_{\Out,i}}).
\end{align}
We define the \textit{flavour group}
\begin{align}\label{flavour group}
\sF:=G_{\mathrm{fram}}^{\mathrm{in}}\times G_{\mathrm{fram}}^{\mathrm{out}}\times G_{\mathrm{edge}}. 
\end{align}
It is easy to see that $\mathsf F\cong \Aut_{G}(R(Q,\mathbf{v},\underline{\bd}))$.

\begin{Definition} 
We say that a framing $\underline{\bd}$ is \textit{symmetric} if $\bd_{\In}=\bd_{\Out}=\bd$. In this case, we simplify the notations as 
\begin{equation}\label{equ on sym qu}R(\bv,\bd)=R(Q,\bv,\bd)= R(Q,\mathbf{v},\underline{\bd}), \quad 
\cM_\theta(\bv,\bd)=\cM_\theta(Q,\bv,\bd)=\mathcal M_{\theta}(Q,\mathbf v,\underline{\bd}), \end{equation} 
and we define $G_{\mathrm{fram}}^{\mathrm{\diag}}\subseteq  \mathsf F$ to be the diagonal subgroup of $G_{\mathrm{fram}}^{\mathrm{in}}\times G_{\mathrm{fram}}^{\mathrm{out}}$.
\end{Definition}

\begin{Definition}\label{def of sym quiver}
We say $Q$ is \textit{symmetric} if its adjacency matrix $ (\mathsf Q_{ij})_{i,j\in Q_0}$ is symmetric. 
For a symmetric quiver $Q$, we call the associated quiver variety \textit{symmetric quiver variety} (SQV) if the framing is symmetric.
\end{Definition}

\begin{Definition}\label{def:self-dual tori}
Let $(Q,\mathbf v,\mathbf{d})$ be a symmetric quiver with symmetric framing as above. Suppose that there exists a torus $\mathsf H$ together with a linear action of $\mathsf H$ on $R(Q,\mathbf v,\mathbf{d})$. We say that a linear action of a torus $\mathsf H$ on $R(Q,\mathbf v,\mathbf{d})$ is \textit{self-dual} if the $\mathsf H$-action commutes with the gauge group $G$-action, and $R(Q,\mathbf v,\mathbf{d})$ is self-dual as $(G\times \mathsf H)$-representation. 

We say that a torus action $\mathsf H$ on the symmetric quiver variety $\cM_\theta(Q,\mathbf v,\mathbf{d})$ is \textit{self-dual} if it is induced from a self-dual action of $\mathsf H$ on $R(Q,\mathbf v,\mathbf{d})$.
\end{Definition}

The following lemma is elementary, and we leave the proof to interested readers.

\begin{Lemma}\label{self-dual torus}
Let $(Q,\mathbf v,\mathbf{d})$ be a symmetric quiver with symmetric framing as above. Suppose that there exists a torus $\mathsf H$ together with a self-dual action on $R(Q,\mathbf v,\mathbf{d})$. Then there exists a decomposition of the set of arrows $Q_1=\mathcal A\sqcup \mathcal A^*\sqcup \mathcal E$ with the following properties
\begin{enumerate}
\setlength{\parskip}{1ex}
    \item There is a one-to-one correspondence between $\mathcal A$ and $\mathcal A^*$, which sends an arrow $a\in \mathcal A$ to an arrow $a^*\in \mathcal A^*$ with opposite direction, i.e. $t(a)=h(a^*)$ and $h(a)=t(a^*)$;
    \item $\mathcal E$ is a set of edge loops, i.e. $t(\varepsilon)=h(\varepsilon)$ for all $\varepsilon\in \mathcal E$;
    \item Consider the torus $(\bC^*)^{\mathcal A}$ which maps $(X_a,X_{a^*})_{a\in \mathcal A}$ to $(t_aX_a,t_a^{-1}X_{a^*})_{a\in \mathcal A}$ and acts on $(X_\varepsilon)_{\varepsilon\in \mathcal E}$ and $(A_i,B_i)_{i\in Q_0}$ trivially, then there exists a maximal torus $T_{\mathrm{fram}}\subseteq  G_{\mathrm{fram}}^{\mathrm{diag}}$ and an element $g\in G_{\mathrm{edge}}$ such that $\mathsf H$-action factors through the subgroup $T_{\mathrm{fram}}\times g(\bC^*)^{\mathcal A}g^{-1}\subseteq  \mathsf F$.
\end{enumerate}
\end{Lemma}


The $\sA$-fixed loci of a symmetric quiver variety are symmetric quiver varieties.

\begin{Lemma}
\cite[Lem.~A.7]{COZZ1}
\label{fix pts of sym quiv var}
Suppose that $X=\mathcal M_{\theta}(Q,\mathbf v,\mathbf{d})$ is a symmetric quiver variety, and $\sA$ is a torus with a self-dual action on $X$. Let $\sigma$ be a cocharacter of $\sA$, then the $\sigma$-fixed points locus $X^\sigma$ is a disjoint union of symmetric quiver varieties. Moreover, the induced action of $\sA$ on $X^\sigma$ is self-dual.
\end{Lemma}

\section{Proof of Theorem \ref{main thm}}\label{sect on proof}

\begin{proof}[Proof of Theorem \ref{main thm}]
The proof of the theorem occupies this section, specifically given as follows:

(1) follows by combining Proposition \ref{X tilde as Nakajima quiv}, Proposition \ref{prop of tau}, Remark \ref{pi_X is a symp res}.

(2) follows from Lemma \ref{open part Cartesian}, Lemma \ref{X^circ is affine}.

(3) follows from Lemma \ref{lem proper birational} and Proposition \ref{prop smallness}.
\end{proof}

\subsection{Family of smooth symplectic varieties}
\begin{Setting}\label{setting of family of symp}
Let $X$ be smooth quasi-projective variety with a torus $\sA$-action, and let $\sigma$ be a cocharacter of $\sA$. Suppose that there exists an $\sA$-invariant two-form $\omega\in \Omega^2(X)$, a smooth connected variety $B$ endowed with trivial $\sA$-action and a smooth and $\sA$-equivariant morphism 
$$\phi\colon X\to B, $$ such that $\forall\,\, b\in B$, the restriction of $\omega$ to $X_b:=\phi^{-1}(b)$ is nondegenerate, i.e. $\omega|_{X_b}$ is a \textit{symplectic structure}.
\end{Setting}

Examples satisfying Setting \ref{setting of family of symp} come from universal deformations of equivariant symplectic resolutions \cite[\S 3.7]{MO},
for instance the deformations of Nakajima quiver varieties as recalled below.

\begin{Example}\label{extended nak quiv}
Consider a quiver $Q=(Q_0,Q_1)$. Take $\mathbf{v},\mathbf{d}\in \bN^{Q_0}$, then the space of framed representations of $Q$ with gauge dimension $\mathbf{v}$ and in-coming framing dimension $\mathbf{d}$ is
\begin{align*}
    M:=R(Q,\bv,\bd,\mathbf 0)=\bigoplus_{a\in Q_1}\underbrace{\Hom(\bC^{\mathbf v_{t(a)}},\bC^{\mathbf v_{h(a)}})}_{X_a}\oplus \bigoplus_{i\in Q_0}\underbrace{\Hom(\bC^{\mathbf{d}_i},\bC^{\mathbf v_i})}_{A_i}.
\end{align*}
The gauge group $G=\prod_{i\in Q_0}\GL(\mathbf v_i)$ naturally acts on $M$ by compositions with maps. This induces a Hamiltonian $G$-action on the cotangent bundle $T^*M$ with moment map $\mu\colon T^*M\to \mathfrak{g}^*$ where $\mathfrak{g}$ is the Lie algebra of $G$. Define the linear subspace $\mathfrak{Z}:=(\mathfrak{g}^*)^G\subseteq  \mathfrak{g}^*$. We note that $$\mathfrak{Z}=(\mathfrak{g}/[\mathfrak{g},\mathfrak{g}])^*\cong \bigoplus_{i\in Q_0}\bC\cdot \mathrm{Id}_i.$$ 
Fix a stability condition $\theta\in \bQ^{Q_0}$, we define the \textit{universally-deformed Nakajima quiver variety} as the GIT quotient:
\begin{equation}\label{equ on univ nak var}
    \widetilde{\cN}_{\theta}(\overline Q,\mathbf v,\mathbf{d}):=\mu^{-1}(\mathfrak{Z})/\!\!/_{\theta} G.
\end{equation}
We assume that $\theta$ is generic such that $\theta$-semistable representations are $\theta$-stable. We assume that $\widetilde{\cN}_{\theta}(\overline Q,\mathbf v,\mathbf{d})$ is nonempty, then the image of $\mu$ contains $\mathfrak{Z}$ by \cite[Prop.~2.2.2]{MO}. Let $\mu^{-1}(\mathfrak{Z})^{s}$ be the $\theta$-stable locus, then $\mu$ is smooth along $\mu^{-1}(\mathfrak{Z})^{s}$ and $G$ acts on $\mu^{-1}(\mathfrak{Z})^{s}$ freely by \cite[Lem.~4.1.7]{Gin}. 

Consider the morphism 
$$\phi\colon \widetilde{\cN}_{\theta}(\overline Q,\mathbf v,\mathbf{d})\to \mathfrak{Z}, $$ 
induced by $\mu|_{\mu^{-1}(\mathfrak{Z})^{s}}\colon \mu^{-1}(\mathfrak{Z})^{s}\to \mathfrak{Z}$. Then $\phi$ is smooth, surjective and
\begin{align*}
    \phi^{-1}(\lambda)\cong \cN_{\theta,\lambda}(\overline Q,\mathbf v,\mathbf{d}):=\mu^{-1}(\lambda)/\!\!/_{\theta} G
\end{align*}
for arbitrary $\lambda\in \mathfrak{Z}$. Define open subset 
\begin{align}\label{open part}
\mathfrak{Z}^\circ := \mathfrak{Z}\setminus \bigcup_{\mathbf u}H_{\mathbf u},
\end{align}
where $\mathbf u\in \bN^{Q_0}$ and $\mathbf u_i\leqslant \mathbf v_i$ for all $i\in Q_0$ and $\mathbf u\neq 0$, and 
$$H_{\mathbf u}=\{\lambda\in \mathfrak{Z}\colon \sum_{i\in Q_0}\lambda_i\mathbf u_i=0\}. $$ 
$\mathfrak{Z}^\circ $ is complement of the union of hyperplanes, in particular it is nonempty. By \cite[Thm.~1.3]{CB}, every element in $\mu^{-1}(\mathfrak{Z}^\circ)$ is a simple quiver representation, in particular every element in $\mu^{-1}(\mathfrak{Z}^\circ)$ is $\theta$-stable for arbitrary $\theta\in \bQ^{Q_0}$. It follows that $\phi^{-1}(\mathfrak{Z}^\circ)\cong \mu^{-1}(\mathfrak{Z}^\circ)/\!\!/ G$ is an affine variety, hence $\phi|_{\phi^{-1}(\mathfrak{Z}^\circ)}$ is an affine morphism.

Let $(\bC^{*})^{Q_1}$ be the torus that acts on $R$ by
\begin{align*}
t_a\cdot(X_a,A_i)=(t_aX_a,A_i).
\end{align*}
Define $G_{\mathrm{fram}}=\prod_{i\in Q_0}\GL(\mathbf{d}_i)$, which naturally acts on framing vector space $W=\bigoplus_{i\in Q_0}\bC^{\mathbf{d}_i}$. Then the $(G_{\mathrm{fram}}\times (\bC^{*})^{Q_1})$-action on $M$ induces a natural $(G_{\mathrm{fram}}\times (\bC^{*})^{Q_1})$-action on $T^*M$ which preserves the symplectic structure. Moreover the moment map $\mu\colon T^*M\to \mathfrak{g}^*$ is $\left(G_{\mathrm{fram}}\times (\bC^{*})^{Q_1}\right)$-invariant. We denote $T_{\mathrm{fram}}\subseteq  G_{\mathrm{fram}}$ to be a maximal torus, and define 
$$\sA=T_{\mathrm{fram}}\times (\bC^{*})^{Q_1}. $$
It is known that there is a $(G\times \sA)$-invariant symplectic form $\omega_{T^*M}$ on $T^*M$ \cite[\S 1]{CB}. Then $\omega_{T^*M}$ induces a two-form $\omega$ on $\widetilde{\cN}_{\theta}(\overline Q,\mathbf v,\mathbf{d})$, which is $\sA$-invariant by construction. The restriction of $\omega$ to every fiber $\cN_{\theta,\lambda}(\overline Q,\mathbf v,\mathbf{d})$ is nondegenerate \cite[Lem.~4.1.7]{Gin}.

In summary, $\widetilde{\cN}_{\theta}(\overline Q,\mathbf v,\mathbf{d})$, $\sA=T_{\mathrm{fram}}\times (\bC^{*})^{Q_1}$, $B=\mathfrak{Z}$, and $B^\circ=\mathfrak{Z}^\circ$ 
satisfy the setting of Setting \ref{setting of family of symp}. 
\end{Example}

\begin{Lemma}\label{symp res has cond star}
In Setting \ref{setting of family of symp}, assume there exists an open dense subset $B^\circ\subseteq  B$ such that $\phi|_{X^\circ}\colon X^\circ\to B^\circ$ is an affine morphism where $X^\circ:=\phi^{-1}(B^\circ)$. 
Then $\phi|_{X^{\sigma}}\colon X^{\sigma}\to B$ is smooth and the restriction of $\omega$ to $X^{\sigma}_b:=(\phi|_{X^{\sigma}})^{-1}(b)$ is nondegenerate for any $b\in B$, and $\phi|_{X^{\sigma}}\colon X^{\sigma}\cap X^\circ\to B^\circ$ is affine.
\end{Lemma}

\begin{proof}
For every $x\in X^\sigma$, the tangent map $d\phi_x\colon T_xX\to T_{\phi(x)}B$ is surjective by the assumption that $\phi\colon X\to B$ is smooth. Since $\phi$ is $\sigma$-equivariant, 
$d\phi_x$ is also $\sigma$-equivariant. Then it follows that $d\phi_x$ maps $T_xX^\sigma=(T_xX)^\sigma$ surjectively onto $T_{\phi(x)}B$; thus the restriction of $\phi$ to the $\sigma$-fixed locus $\phi|_{X^{\sigma}}\colon X^\sigma\to B$ is also smooth. Since $X^{\sigma}\cap X^\circ$ is a closed subvariety of $X^\circ$, the map $\phi|_{X^{\sigma}}\colon X^{\sigma}\cap X^\circ\to B^\circ$ is affine. 
\end{proof}

\subsection{A Grothendieck-Springer type resolution for symmetric quiver varieties}\label{sec sym quiv vs nak quiv}


Throughout this section, we take $Q$ to be a symmetric quiver and $\mathbf v,\mathbf{d}\in \bN^{Q_0}$ such that $\mathbf{v}_i\neq 0$ for all $i\in Q_0$ (otherwise we remove any node $i$ such that $\mathbf v_i=0$) and $\mathbf{d}\neq 0$. As in Definition \ref{def of sym quiver}, we have the space of symmetrically framed representations $R(Q,\mathbf v,\mathbf{d})$ with a linear $G=\prod_{i\in Q_0}\GL(\mathbf v_i)$ action and the quiver variety $$X:=\mathcal M_\theta(Q,\mathbf v,\mathbf{d})=R(Q,\mathbf v,\mathbf{d})/\!\!/_\theta G.$$ We fix a decomposition of set of arrows $Q_1=\mathcal A\sqcup \mathcal A^*\sqcup \mathcal E$ then take $\sA=T_{\mathrm{fram}}\times (\bC^*)^{\mathcal A}$ as in the Lemma \ref{self-dual torus}.


Let us rewrite the space of representations $R(Q,\mathbf v,\mathbf{d})$ in a way that resembles the construction of Nakajima quiver variety. Consider the following two new quivers
\begin{enumerate}
\setlength{\parskip}{1ex}
    \item $Q^{\add}$ is obtained from $Q$ by adding one loop to each node, so $Q^{\add}_0=Q_0$,  $Q^{\add}_1=Q_1\sqcup Q_0$;
    \item $Q^{\rem}$ which is obtained from $Q$ by removing all loops in the set $\mathcal E$ to each node, so $Q^{\rem}_0=Q_0$,  $Q^{\rem}_1=Q_1\setminus \mathcal E$.
\end{enumerate}
We decompose the set of arrows $Q^{\add}_1=\mathcal A\sqcup \mathcal A^*\sqcup \mathcal E^{\add}$, where $\mathcal E^{\add}:=\mathcal E\sqcup Q_0$ is the union of the original set of loops $\mathcal E$ and all newly-added loops. Then the space of representations is
\begin{align}\label{rep of Q add}
    R(Q^{\add},\mathbf v,\mathbf{d})=R(Q^{\rem},\mathbf v,\mathbf{d})\oplus \bigoplus_{\varepsilon\in \mathcal E^{\add}}\mathfrak{gl}(\mathbf v_{h(\varepsilon)})^*.
\end{align}
We note that $R(Q^{\rem},\mathbf v,\mathbf{d})\cong T^*M$ where 
\begin{align*}
    M=\bigoplus_{a\in \mathcal A}\Hom(\bC^{\mathbf v_{t(a)}},\bC^{\mathbf v_{h(a)}})\oplus \bigoplus_{i\in Q_0}\Hom(\bC^{\mathbf{d}_i},\bC^{\mathbf v_i}).
\end{align*}
In particular, $R(Q^{\rem},\mathbf v,\mathbf{d})$ is endowed with the natural cotangent bundle symplectic structure. Note that $\mathfrak{gl}(\mathbf v_{i})^*$ is endowed with the Kirillov-Kostant-Souriau Poisson structure. Therefore $R(Q^{\add},\mathbf v,\mathbf{d})$ is endowed with a natural Poisson structure, defined as the direct sum of Poisson structures on each of its summand. 

The group $(G\times \sA)$ acts on $R(Q^{\add},\mathbf v,\mathbf{d})$ and $R(Q^{\rem},\mathbf v,\mathbf{d})$ naturally by compositions with maps. Then the $G$-action on $R(Q^{\add},\mathbf v,\mathbf{d})$ is Hamiltonian with moment map
\begin{align*}
    \mu^{\add}\colon R(Q^{\add},\mathbf v,\mathbf{d})\to \mathfrak{g}^*, \quad \mu^{\add}=(\mu^{\rem},\mathrm{pr}), 
\end{align*}
where $\mu^{\rem}\colon R(Q^{\rem},\mathbf v,\mathbf{d})\to \mathfrak{g}^*$ is the moment map for $R(Q^{\rem},\mathbf v,\mathbf{d})$ and $\mathrm{pr}$ maps every $\mathfrak{gl}(\mathbf v_{i})^*$ identically to itself.
\begin{Lemma}
There is $(G\times \sA)$-equivariant isomorphism $R(Q,\mathbf v,\mathbf{d})\cong (\mu^{\add})^{-1}(0)$.
\end{Lemma}

\begin{proof}
By writing 
$$R(Q,\mathbf v,\mathbf{d})=R(Q^{\rem},\mathbf v,\mathbf{d})\oplus \bigoplus_{\varepsilon\in \mathcal E}\mathfrak{gl}(\mathbf v_{h(\varepsilon)})^*,$$ 
$R(Q,\mathbf v,\mathbf{d})$ is equipped with natural Poisson structure and the $(G\times \sA)$-action is Hamiltonian with moment map $\mu\colon R(Q,\mathbf v,\mathbf{d})\to \mathfrak{g}^*$. Then with respect to the decomposition 
$$R(Q^{\add},\mathbf v,\mathbf{d})=R(Q,\mathbf v,\mathbf{d})\oplus \bigoplus_{i\in Q_0}\mathfrak{gl}(\mathbf v_{i})^*=R(Q,\mathbf v,\mathbf{d})\oplus \mathfrak{g}^*,$$ 
$$\mu^{\add}(x,y)=\mu(x)+y. $$
The the equation $\mu^{\add}=0$ removes the $\mathfrak{g}^*$ component in $R(Q^{\add},\mathbf v,\mathbf{d})$ and leaves us with $R(Q,\mathbf v,\mathbf{d})$, i.e. $R(Q,\mathbf v,\mathbf{d})\cong (\mu^{\add})^{-1}(0)$. This isomorphism is obviously $(G\times \sA)$-equivariant.
\end{proof}


Let $L$ be a complex reductive group and we fix a Borel subgroup $B$ with a maximal torus $H\subseteq  B$. Let the Weyl group of $L$ be $W$. Denote the corresponding Lie algebras $\mathfrak{l}:=\Lie(L),\mathfrak{b}:=\Lie(B),\mathfrak{h}:=\Lie(H)$.
We have the the following commutative diagram (\cite[(3.2.6)]{CG}):
\begin{equation}\label{Grothendieck-Springer resolution}
\xymatrix{
\widetilde{\mathfrak{l}}^*:=L\times^B \mathfrak{b}^* \ar[r]^-{\nu_{\mathfrak{l}}} \ar[d]_{\pi_{\mathfrak{l}}} & \mathfrak{h}^* \ar[d]^{p_{\mathfrak{l}}}\\
\mathfrak{l}^* \ar[r]^-{\rho_{\mathfrak{l}}} & \mathfrak{h}^*/W\cong \mathfrak{l}^*/L.
}
\end{equation}
Here $p_{\mathfrak{l}}$ and $\rho_{\mathfrak{l}}$ are quotient maps, and $\pi_{\mathfrak{l}}$ maps $(g,b)\in L\times^B \mathfrak{b}^*$ to $g\cdot b\in \mathfrak{l}^*$, and $\nu_{\mathfrak{l}}$ maps $(g,b)\in L\times^B \mathfrak{b}^*$ to $b|_{\mathfrak{h}}\in \mathfrak{h}^*$. We note that $\pi_{\mathfrak{l}}$ is equivariant with respect to the natural $L$-actions.

\begin{Theorem}[{\cite[\S II 4.7]{Slo}}]\label{GS resolution}
In the above commutative diagram 
\begin{enumerate}
\setlength{\parskip}{1ex}
    \item $\nu_{\mathfrak{l}}$ is a smooth morphism;
    \item $p_{\mathfrak{l}}$ is a finite and surjective morphism;
    \item $\pi_{\mathfrak{l}}$ is a proper and surjective morphism;
    \item $\rho_{\mathfrak{l}}$ is a flat morphism;
    \item $\forall\, t\in \mathfrak{h}^*$, the induced morphism $\nu_{\mathfrak{l}}^{-1}(t)\to \rho_{\mathfrak{l}}^{-1}(p_{\mathfrak{l}}(t))$ is a resolution of singularities.
\end{enumerate}
\end{Theorem}
In above, $\pi_{\mathfrak{l}}\colon \widetilde{\mathfrak{l}}^*\to \mathfrak{l}^*$ is called the \textit{Grothendieck-Springer resolution} of $\mathfrak{l}^*$.
We apply this resolution to each $\mathfrak{gl}(\mathbf v_i)^*$ factor in $R(Q^{\add},\mathbf v,\mathbf{d})$, namely we define
\begin{align}\label{rep of Q add_resolved}
    \widetilde{R}(Q^{\add},\mathbf v,\mathbf{d})=R(Q^{\rem},\mathbf v,\mathbf{d})\times \prod_{\varepsilon\in \mathcal E^{\add}} \widetilde{\mathfrak{gl}(\mathbf v_{h(\varepsilon)})}^*.
\end{align}
Then we get a $(G\times \sA)$-equivariant proper and surjective and generically finite morphism 
$$\pi_R\colon \widetilde{R}(Q^{\add},\mathbf v,\mathbf{d})\to {R}(Q^{\add},\mathbf v,\mathbf{d}). $$  
It follows that the induced morphism on the quotients:
\begin{align*}
    \pi_{X}\colon \widetilde{X}:=\pi_R^{-1}\left(\left((\mu^{\add})^{-1}(0)\right)^{s}\right)/\!\!/ G\to\left((\mu^{\add})^{-1}(0)\right)^{s}/\!\!/ G\cong \mathcal M_{\theta}(Q,\mathbf v,\mathbf{d})=X
\end{align*}
is an $\sA$-equivariant proper and surjective morphism. 

Moreover, $\pi_{X}\colon \widetilde{X}\to X$ fits into the following commutative diagram of $\sA$-equivariant morphisms:
\begin{equation}\label{Gro-Spr res for quiv}
\xymatrix{
\widetilde{X} \ar[r]^{\widetilde{\tau}} \ar[d]_{\pi_X} & \mathcal{H} \ar[d]^{p}\\
X \ar[r]^{\tau\,\,\,} & \mathcal{H}/\Gamma\,.
}
\end{equation}
Here $\mathcal H=\prod_{\varepsilon\in \mathcal E^{\add}}\bC^{\mathbf v_{h(\varepsilon)}}$ is the direct sum of the dual of the Cartan subalgebras of $\mathfrak{gl}(\mathbf v_{h(\varepsilon)})$, $\Gamma=\prod_{\varepsilon\in \mathcal E^{\add}}\mathfrak{S}_{\mathbf v_{h(\varepsilon)}}$ is the product of Weyl groups, $p\colon \mathcal H\to \mathcal H/\Gamma$ is the quotient map, $\widetilde{\tau}$ and $\tau$ are induced from the compositions of the horizontal maps in the following diagram
\begin{equation}\label{Gro-Spr res for quiv before quot}
\xymatrix{
\widetilde{R}(Q^{\add},\mathbf v,\mathbf{d}) \ar[r] \ar[d]_{\pi_R} \ar@{}[dr]|{\Box}  & \prod_{\varepsilon\in \mathcal E^{\add}} \widetilde{\mathfrak{gl}(\mathbf v_{h(\varepsilon)})}^* \ar[d]_{\pi} \ar[r]^{\nu} & \prod_{\varepsilon\in \mathcal E^{\add}}\bC^{\mathbf v_{h(\varepsilon)}} \ar[d]^{p}\\
{R}(Q^{\add},\mathbf v,\mathbf{d}) \ar[r] & \prod_{\varepsilon\in \mathcal E^{\add}} \mathfrak{gl}(\mathbf v_{h(\varepsilon)})^* \ar[r]^{\rho\quad} & \prod_{\varepsilon\in \mathcal E^{\add}}\bC^{\mathbf v_{h(\varepsilon)}}/\mathfrak{S}_{\mathbf v_{h(\varepsilon)}}\,,
}
\end{equation}
where the left square is Cartesian, and horizontal maps there are projections to components, and the square on the right-hand-side is the the diagram \eqref{Grothendieck-Springer resolution} for $\mathfrak{l}=\bigoplus_{\varepsilon\in \mathcal E^{\add}} \mathfrak{gl}(\mathbf v_{h(\varepsilon)})$.

Our next goal is to show that $\widetilde{\tau}\colon \widetilde{X} \to \mathcal{H} $ satisfies conditions in Setting \ref{setting of family of symp}, 
by relating $\widetilde{X}$ to a universally-deformed Nakajima quiver variety, see Proposition \ref{X tilde as Nakajima quiv}.
In particular this will imply that $\widetilde{X}$ is smooth, which is not obvious from its definition.


\subsection{An auxiliary quiver}
We construct a new quiver $Q^{\aux}$ with new dimension vectors $\mathbf v^{\aux},\mathbf{d}^{\aux}$ from $(Q^{\add},\mathbf v,\mathbf{d})$ by the following procedure. For every loop $\varepsilon\in \mathcal E^{\add}$, and let $n=\mathbf v_{h(\varepsilon)}$, we replace the $\varepsilon$ by a doubled $A_{n-1}$ quiver:
\begin{equation}\label{aux quiv}
\xymatrix{
 \circled{n} \ar@(dr,ur)_{\varepsilon} 
& \ar@{~>}[r] & \qquad \fillcircled{n}    \ar@<.5ex>[r]  & \ar@<.5ex>[l] \fillcircled{n-1}  \ar@<.5ex>[r]  & \ar@<.5ex>[l] \cdots  \ar@<.5ex>[r]  & \ar@<.5ex>[l] \fillcircled{1}
}.
\end{equation}
$\mathbf v^{\aux}$ is determined by the above replacement, and we define $\mathbf{d}^{\aux}$ to be equal to $\mathbf{d}$ on the original node of $Q^{\add}$ and extend to newly-added nodes by zero. We note that 
\begin{align}\label{rep of Q aux}
    R(Q^{\aux},\mathbf v^{\aux},\mathbf{d}^{\aux})=R(Q^{\rem},\mathbf v,\mathbf{d})\oplus T^*M^\leg,
\end{align}
where $M^\leg$ is the space of linear maps corresponding to the left-pointing arrows in the right-hand-side of \eqref{aux quiv}, in particular $R(Q^{\aux},\mathbf v^{\aux},\mathbf{d}^{\aux})$ is a symplectic vector space. In fact, $Q^{\aux}$ is the doubling $\overline{Q'}$ of a quiver $Q'$.

We extend the $\sA$-action on $R(Q^{\rem},\mathbf v,\mathbf{d})$ to $R(Q^{\aux},\mathbf v^{\aux},\mathbf{d}^{\aux})$ by requiring that $\sA$ acts trivially on $T^*M^\leg$. The gauge group for $(Q^{\aux},\mathbf v^{\aux},\mathbf{d}^{\aux})$ is $G^{\aux}=\prod_{i\in Q^{\aux}_0}\GL(\mathbf{v}^{\aux}_i)$, and its Lie algebra is denoted by $\mathfrak{g}^{\aux}$. The natural $G^{\aux}$-action on $R(Q^{\aux},\mathbf v^{\aux},\mathbf{d}^{\aux})$ is Hamiltonian with moment map
\begin{align*}
    \mu^{\aux}\colon R(Q^{\aux},\mathbf v^{\aux},\mathbf{d}^{\aux})\to (\mathfrak{g}^{\aux})^*.
\end{align*}
Note that $\mathfrak{g}^{\aux}=\mathfrak{g}\oplus \mathfrak{g}^\leg$, where $\mathfrak{g}^\leg$ is the Lie algebra of the gauge group corresponding to the newly-added nodes:
$$G^\leg=\prod_{i\in Q^{\aux}_0\setminus Q_0}\GL(\mathbf{v}^{\aux}_i).$$
Then we can write $\mu^{\aux}$ in components $\mu^{\aux}=(\mu^{\aux}_{\mathrm{res}},\mu^\leg)$, where 
$$\mu^{\aux}_{\mathrm{res}}\colon R(Q^{\aux},\mathbf v^{\aux},\mathbf{d}^{\aux})\to \mathfrak{g}^*, \quad 
\mu^\leg\colon R(Q^{\aux},\mathbf v^{\aux},\mathbf{d}^{\aux})\to (\mathfrak{g}^\leg)^*$$ are the corresponding components.

Define the linear subspaces
\begin{align*}
    \mathfrak{Z}:=(\mathfrak{g}^*)^G\subseteq  \mathfrak{g}^*,\quad \mathfrak{Z}^\leg:=((\mathfrak{g}^\leg)^*)^G\subseteq  (\mathfrak{g}^\leg)^*,\quad \mathfrak{Z}^{\aux}:=\mathfrak{Z}\oplus \mathfrak{Z}^\leg\subseteq  (\mathfrak{g}^\aux)^*.
\end{align*}
We note that $\mathfrak{Z}\cong \bigoplus_{i\in Q_0}\bC\cdot \mathrm{Id}_i$ and $\mathfrak{Z}^\leg\cong \bigoplus_{i\in Q^\aux_0\setminus Q_0}\bC\cdot \mathrm{Id}_i$.

\begin{Definition}\label{map A_n leg to loop}
We define a morphism 
\begin{align}\label{frak p}
\mathfrak{p}\colon R(Q^{\aux},\mathbf v^{\aux},\mathbf{d}^{\aux})\times \bC^{\mathcal E^\add}\to R(Q^{\add},\mathbf v,\mathbf{d})
\end{align}
as follows. In view of decompositions \eqref{rep of Q add} and \eqref{rep of Q aux}, let $\mathfrak{p}$ be identity map on the component $R(Q^{\rem},\mathbf v,\mathbf{d})$, and for each loop $\varepsilon\in \mathcal E^\add$, let $\mathfrak{p}$ maps the corresponding $A_{n-1}$ leg by:
\begin{equation}
\left(\xymatrix{
  \fillcircled{n}    \ar@<.5ex>[r]^{D^\varepsilon_{n-1}}  & \ar@<.5ex>[l]^{C^\varepsilon_{n-1}} \fillcircled{n-1}  \ar@<.5ex>[r]^{D^\varepsilon_{n-2}}  & \ar@<.5ex>[l]^{C^\varepsilon_{n-2}} \cdots  \ar@<.5ex>[r]^{D^\varepsilon_{1}\quad}  & \ar@<.5ex>[l]^{C^\varepsilon_{1}\quad} \fillcircled{1} },\,\,\,  t_\varepsilon\in \bC\right)
 \xymatrix{ 
  \,\,\, \ar@{|->}[r] &  \quad\circled{n} \ar@(dr,ur)_{X_\varepsilon:=C^\varepsilon_{n-1}D^\varepsilon_{n-1}+t_\varepsilon\cdot\id} 
}.
\end{equation}
We also define
\begin{align}\label{frak p bar}
\overline{\mathfrak{p}}:(\mu^\leg)^{-1}(\mathfrak{Z}^\leg)\times \bC^{\mathcal E^\add}\to R(Q^{\add},\mathbf v,\mathbf{d})
\end{align}
to be the restriction of $\mathfrak{p}$ to the subvariety $(\mu^\leg)^{-1}(\mathfrak{Z}^\leg)\times \bC^{\mathcal E^\add}$.
\end{Definition}

We note that $\mathfrak{p}$ is $(G^\aux\times\sA)$-equivariant, where $(G^\aux\times\sA)$ acts $R(Q^{\add},\mathbf v,\mathbf{d})$ via projection to $G\times\sA$.

\begin{Lemma}\label{compare moment maps}
There is an $(G^\aux\times\sA)$-equivariant isomorphism 
$$\mathfrak{p}^{-1}((\mu^\add)^{-1}(0))\cong (\mu^{\aux}_{\mathrm{res}})^{-1}(\mathfrak{Z})\times \mathcal L, $$ 
where 
\begin{equation}\label{equ on vs L}
\mathcal L:=\mathrm{ker}\left(\mathrm{sum}\colon \bC^{\mathcal E^\add}\to \mathfrak{Z} \right)
\end{equation}
 is the $Q_0$-graded linear subspace of $\bC^{\mathcal E^\add}$ defined as the kernel of the following surjective map:
\begin{align*}
\mathrm{sum}\colon  \bC^{\mathcal E^\add}\to \mathfrak{Z}, \quad  (t_\varepsilon)_{\epsilon\in \mathcal E^\add}\mapsto \left(\sum_{h(\varepsilon)=i}t_\varepsilon\right)_{i\in Q_0}.
\end{align*}
\end{Lemma}

\begin{proof}
Let us fix a node $i\in Q_0$, then the equation $\mu^\add_i\circ \mathfrak{p}=0$ is equivalent to
\begin{align*}
    \mu^\rem_i+\sum_{\varepsilon\in \mathcal{E}^\add,h(\varepsilon)=i}(C^\varepsilon_{n-1}D^\varepsilon_{n-1}+t_\epsilon\cdot\id)=0.
\end{align*}
Notice that 
\begin{align*}
    \mu^{\aux}_{\mathrm{res},i}=\mu^\rem_i+\sum_{\varepsilon\in \mathcal{E}^\add,h(\varepsilon)=i}C^\varepsilon_{n-1}D^\varepsilon_{n-1},
\end{align*}
thus the equation $\mu^\add_i\circ \mathfrak{p}=0$ can be rewritten as
\begin{align*}
    \mu^{\aux}_{\mathrm{res},i}+\sum_{\varepsilon\in \mathcal{E}^\add,h(\varepsilon)=i}t_\epsilon\cdot\id=0,
\end{align*}
then the lemma follows.
\end{proof}

\begin{Definition}
We say that a representation $x\in R(Q^{\aux},\mathbf v^{\aux},\mathbf{d}^{\aux})$ is \textit{leg-stable} if all linear maps corresponding to the left-pointing arrows in the right-hand-side of \eqref{aux quiv} are injective. The subset of all leg-stable representations is denoted by $R(Q^{\aux},\mathbf v^{\aux},\mathbf{d}^{\aux})^{\leg\emph{-}s}$.\footnote{It is an Zariski open subset of $R(Q^{\aux},\mathbf v^{\aux},\mathbf{d}^{\aux})$.}
\end{Definition}
The motivation for introducing leg stability condition is to state Lemma \ref{T[SU(n)] quiver}. Consider the quiver
\begin{equation*}
Q^n\colon \qquad
\xymatrix{
  \boxed{n}    \ar@<.5ex>[r]^{D_{n-1}}  & \ar@<.5ex>[l]^{C_{n-1}} \fillcircled{n-1}  \ar@<.5ex>[r]^{D_{n-2}}  & \ar@<.5ex>[l]^{C_{n-2}} \cdots  \ar@<.5ex>[r]^{D_{1}\quad}  & \ar@<.5ex>[l]^{C_{1}\quad} \fillcircled{1} 
},
\end{equation*}
where the gauge and framing dimension vectors are $\mathbf v^n$ (with $\mathbf v^n_i=n-i$), and  $\mathbf{d}^n$ (with $\mathbf{d}^n_i=n\cdot \delta_{i,1}$). 
In the following discussions, we take the stability condition $\zeta$ to be $\zeta_i=1$ for every node $i$. Denote the moment map for the above quiver to be $\mu^n$, and denote the space $\mathfrak{Z}$ for the quiver $Q^n$ to be $\mathfrak{Z}^n$. Then a point in $\mu^n(\mathfrak{Z}^n)$ is a $\zeta$-semistable (equivalently, $\zeta$-stable) $Q^n$-representation if and only if $C_i$ are injective for all $1\leqslant i\leqslant n-1$. 

Consider the morphism in Example \ref{extended nak quiv}:
$$\phi\colon \widetilde{\cN}_\zeta(Q^n,\mathbf v^n,\mathbf{d}^n)\to \mathfrak{Z}^n, $$
where $\widetilde{\cN}_\zeta(Q^n,\mathbf v^n,\mathbf{d}^n)$ is the universally deformed Nakajima variety \eqref{equ on univ nak var}.
For $\lambda=(\lambda_1,\cdots,\lambda_{n-1})\in \mathfrak{Z}^n$, $\phi^{-1}(\lambda)$ is the moduli space   
\begin{align*}
    \left\{(\bC^n=V_0\supset V_1\supset \cdots \supset V_{n-1}\supset V_n=0),\, X\in \mathrm{End}(\bC^n)\:\bigg\rvert\: \substack{\text{$V_i$ is a dim $n-i$ linear subspace, and $X(V_0)\subseteq  V_{1}$, and}\\ \text{for $1\leqslant i\leqslant n-1$, $X(V_i)\subseteq  V_{i}$ and $\overline{X}_i\colon V_{i}/V_{i+1}\to V_{i}/V_{i+1}$} \\ \text{ equals to $\sum_{j=1}^i \lambda_j\cdot\id$}} \right\}.
\end{align*}
\begin{Lemma}\label{T[SU(n)] quiver}
Let $\nu\colon \widetilde{\mathfrak{gl}(n)}^*\to \mathfrak{h}^*$ be the morphism for $L=\GL(n)$ in the Grothendieck-Springer resolution \eqref{Grothendieck-Springer resolution}. Then there exists a commutative diagram
\begin{equation*}
\xymatrix{
\widetilde{\mathfrak{gl}(n)}^* \ar[r]^{\nu} \ar[d]_{\cong} & \mathfrak{h}^* \ar[d]^{\cong}\\
\widetilde{\cN}_\zeta(Q^n,\mathbf v^n,\mathbf{d}^n)\times \bC \ar[r]^{\qquad\,\,\, \phi\times\id} & \mathfrak{Z}^n\times \bC\,,
}
\end{equation*}
with vertical arrows being isomorphisms.
\end{Lemma}

\begin{proof}
Write $\mathfrak{h}^*=\bC^n$. We observe that $\nu^{-1}(r_1,\cdots,r_n)$ is the moduli space of 
\begin{align*}
    \left\{(\bC^n=V_0\supset V_1\supset \cdots \supset V_{n-1}\supset V_n=0), M\in \mathrm{End}(\bC^n)\:\bigg\rvert\: \substack{\text{$V_i$ is a dim $n-i$ linear subspace, and for $1\leqslant i\leqslant n-1$}\\ \text{$M(V_i)\subseteq  V_{i}$ and $\overline{M}_i\colon V_{i}/V_{i+1}\to V_{i}/V_{i+1}$ equals to $r_{i+1}\cdot\id$} } \right\}.
\end{align*}
We define a linear isomorphism 
$$\mathfrak{Z}^n\times \bC\cong \mathfrak{h}^*$$ 
$$((\lambda_1,\cdots,\lambda_{n-1}), t)\mapsto (r_1,\cdots,r_n), \quad \mathrm{where}\, \,r_i=t+\sum_{j=1}^{i-1}\lambda_j. $$  
Then the desired isomorphism is given by: 
\begin{equation*}
\widetilde{\cN}_\zeta(Q^n,\mathbf v^n,\mathbf{d}^n)\times \bC \cong \widetilde{\mathfrak{gl}(n)}^*, \quad (\{(V_i,X)\},t) \mapsto \{(V_i,X+t\cdot \id)\}. 
\qedhere
\end{equation*}
\end{proof}
Applying Lemma \ref{T[SU(n)] quiver} to every leg in the quiver $Q^\aux$, and we get the following.

\begin{Lemma}\label{Gro-Spr for quiv as legs}
Let $\overline\nu\colon \widetilde{R}(Q^{\add},\mathbf v,\mathbf{d}) \to \mathcal H$ be the composition of the top horizontal arrows in the diagram \eqref{Gro-Spr res for quiv before quot}. Then there exists a commutative diagram of $(G\times \sA)$-equivariant morphism
\begin{equation*}
\xymatrix{
&\ar[dl]_{\pi_R}\widetilde{R}(Q^{\add},\mathbf v,\mathbf{d}) \ar[r]^{\quad\overline\nu} \ar[dd]_{\cong} & \mathcal H \ar[dd]^{\cong}\\
{R}(Q^{\add},\mathbf v,\mathbf{d}) & & \\
&\ar[ul]_{\mathfrak{p}'}\left((\mu^\leg)^{-1}(\mathfrak{Z}^\leg)^{\leg\emph{-}s}/\!\!/ G^\leg \right)\times \bC^{\mathcal E^\leg} \ar[r]^-{\phi\times\id} & \mathfrak{Z}^\leg\times \bC^{\mathcal E^\leg}\,,
}
\end{equation*}
with vertical arrows being isomorphisms. Here $\mathfrak{p}'$ is induced from $\overline{\mathfrak{p}}$ by taking the $G^\leg$ quotient.
\end{Lemma}

\begin{Definition}
For a stability condition $\xi\in \bQ^{Q_0}$ and a number $\delta\in \bQ$, we define $\xi^\delta\in \bQ^{Q^\aux_0}$ as follows
\begin{align*}
    \xi^\delta_i=\begin{cases}
        \xi_i\,,& \text{ if }i\in Q_0,\\
        \delta\,, & \text{ if }i\in Q^\aux_0\setminus Q_0.
    \end{cases}
\end{align*}
\end{Definition}

\begin{Lemma}\label{compare stabilities}
For every $\xi\in \bQ^{Q_0}$, there exists $\delta>0$ such that the following hold simultaneously
\begin{align}\label{stable loci inclusion}
    \overline{\mathfrak{p}}^{-1}(R(Q^{\add},\mathbf v,\mathbf{d})^{\xi\emph{-}s})\cap \left((\mu^\leg)^{-1}(\mathfrak{Z}^\leg)^{\leg\emph{-}s}\times \bC^{\mathcal E^\leg}\right)\subseteq  (\mu^\leg)^{-1}(\mathfrak{Z}^\leg)^{\xi^\delta\emph{-}s}\times \bC^{\mathcal E^\leg},
\end{align}
\begin{align}\label{semi-stable loci inclusion}
    \overline{\mathfrak{p}}^{-1}(R(Q^{\add},\mathbf v,\mathbf{d})^{\xi\emph{-}ss})\cap \left((\mu^\leg)^{-1}(\mathfrak{Z}^\leg)^{\leg\emph{-}s}\times \bC^{\mathcal E^\leg}\right)\supseteq (\mu^\leg)^{-1}(\mathfrak{Z}^\leg)^{\xi^\delta\emph{-}ss}\times \bC^{\mathcal E^\leg}.
\end{align}
\end{Lemma}

\begin{proof}
Without loss of generality, let us assume in the below that $\xi\in \bZ^{Q_0}$. A preliminary observation is that for any $\delta>0$ we have inclusion
\begin{align*}
    (\mu^\leg)^{-1}(\mathfrak{Z}^\leg)^{\leg\emph{-}s}\supseteq (\mu^\leg)^{-1}(\mathfrak{Z}^\leg)^{\xi^\delta\emph{-}ss}.
\end{align*}
To proceed, let us recall Crawley-Boevey's trick \cite[\S 1,~Remarks~added~in~April~2000]{CB}. We add a new node $\infty$ to $Q^\add$ and $\mathbf{d}_i$ arrows from $\infty$ to $i\in Q_0$ and $\mathbf{d}_i$ arrows from $i$ to $\infty$; and do the same for $Q^\aux$. We denote the resulting unframed quivers by $Q^\add_{\bullet}$ and $Q^\aux_{\bullet}$. Then we have 
\begin{align}\label{CB trick}
    R(Q^{\add},\mathbf v,\mathbf{d})\cong \mathrm{Rep}(Q^\add_{\bullet},\mathbf v_\bullet),\quad R(Q^{\aux},\mathbf v^{\aux},\mathbf{d}^{\aux})\cong \mathrm{Rep}(Q^{\aux}_{\bullet},\mathbf v^{\aux}_\bullet),
\end{align}
where the restriction of $\mathbf v_\bullet$ to $Q_0$ is $\mathbf v$ and $\mathbf v_{\bullet,\infty}=1$, and restriction of $\mathbf v^{\aux}_\bullet$ to $Q^{\aux}_0$ is $\mathbf v^{\aux}$ and $\mathbf v^{\aux}_{\bullet,\infty}=1$. Moreover, the isomorphisms \eqref{CB trick} respect the stability conditions in the following sense:
\begin{align*}
    R(Q^{\add},\mathbf v,\mathbf{d})^{\xi\emph{-}s}\cong \mathrm{Rep}(Q^\add_{\bullet},\mathbf v_\bullet)^{\xi_\bullet\emph{-}s}&,\quad R(Q^{\aux},\mathbf v^{\aux},\mathbf{d}^{\aux})^{\xi^\delta\emph{-}s}\cong \mathrm{Rep}(Q^{\aux}_{\bullet},\mathbf v^{\aux}_\bullet)^{\xi^\delta_\bullet\emph{-}s},\\
    R(Q^{\add},\mathbf v,\mathbf{d})^{\xi\emph{-}ss}\cong \mathrm{Rep}(Q^\add_{\bullet},\mathbf v_\bullet)^{\xi_\bullet\emph{-}ss}&,\quad R(Q^{\aux},\mathbf v^{\aux},\mathbf{d}^{\aux})^{\xi^\delta\emph{-}ss}\cong \mathrm{Rep}(Q^{\aux}_{\bullet},\mathbf v^{\aux}_\bullet)^{\xi^\delta_\bullet\emph{-}ss},
\end{align*}
where the restriction of $\xi_\bullet$ to $Q_0$ is $\xi$ and $\xi_{\bullet,\infty}=-\sum_{i\in Q_0}\xi_i\mathbf v_i$, and restriction of $\xi^{\delta}_\bullet$ to $Q^{\aux}_0$ is $\xi^{\delta}$ and $\xi^{\delta}_{\bullet,\infty}=-\sum_{i\in Q^{\aux}_0}\xi^{\delta}_i\mathbf v^\aux_i$. By abuse the notations in below, let $\mathfrak{p}$ be a morphism from $\mathrm{Rep}(Q^{\aux}_{\bullet},\mathbf v^{\aux}_\bullet)$ to $\mathrm{Rep}(Q^\add_{\bullet},\mathbf v_\bullet)$.

Fix a $\delta\in \bQ_{>0}$ such that
\begin{align}\label{condition on delta}
    \sum_{i\in Q^\aux_0\setminus Q_0}\delta\cdot \mathbf v^\aux_i<\frac{1}{2}.
\end{align}
We claim that \eqref{stable loci inclusion} and \eqref{semi-stable loci inclusion} hold. To prove them, we introduce notations: 
$(V_i)_{i\in Q^\aux_{\bullet,0}}$ be the $Q^\aux_{\bullet,0}$-graded vector space which is a $Q^\aux_{\bullet}$-representation. It restricts to a $Q^\add_{\bullet}$-representation $(V_i)_{i\in Q^\add_{\bullet,0}}$.

Let us first prove \eqref{semi-stable loci inclusion}. Take an arbitrary point $x\in (\mu^\leg)^{-1}(\mathfrak{Z}^\leg)^{\xi^\delta\emph{-}ss}$, and take arbitrary $(t_\varepsilon\in \bC)_{\varepsilon\in \mathcal E^\add}$. Suppose that the $Q^\add_\bullet$-representation $\overline{\mathfrak{p}}(x,(t_\varepsilon))$ is not $\xi_\bullet$-semistable, i.e. there exists a proper $Q^\add_{\bullet,0}$-graded subspace $(S_i\subseteq  V_i)_{i\in Q^\add_{\bullet,0}}$ which is preserved by all maps in $\overline{\mathfrak{p}}(x,(t_\varepsilon))$ and
\begin{align*}
    \sum_{i\in Q^\add_{\bullet,0}}\xi_{\bullet,i}\dim S_i>0.
\end{align*}
Since $\xi_{\bullet,i}\in \bZ$, we must have
\begin{align}\label{ss implies lower bound}
    \sum_{i\in Q^\add_{\bullet,0}}\xi_{\bullet,i}\dim S_i\geqslant 1.
\end{align}
Then we extend $(S_i\subseteq  V_i)_{i\in Q^\add_{\bullet,0}}$ to a $Q^\aux_{\bullet,0}$-graded subspace $(S_i\subseteq  V_i)_{i\in Q^\aux_{\bullet,0}}$ by setting $S_j$ for $j\notin Q^\add_{\bullet,0}$ to be the image of $S_{i_0}$ along a shortest path from $i_0$ to $j$ where $i_0$ is the unique node in $Q^\add_{\bullet,0}$ with the shortest distance to $j$. It is easy to see that $(S_i\subseteq  V_i)_{i\in Q^\aux_{\bullet,0}}$ is preserved by all maps in the representation $x$. By the $\xi^\delta_\bullet$-semistability of $x$, we have
\begin{align*}
    \sum_{i\in Q^\aux_{\bullet,0}}\xi^\delta_{\bullet,i}\dim S_i\leqslant 0.
\end{align*}
However, we have
\begin{align*}
    \sum_{i\in Q^\aux_{\bullet,0}}\xi^\delta_{\bullet,i}\dim S_i&=\sum_{i\in Q_{0}}\xi_{i}\dim S_i+ \sum_{i\in Q^\aux_0\setminus Q_{0}}\delta\cdot\dim S_i+\xi^\delta_{\bullet,\infty}\dim S_\infty\\
    &=\sum_{i\in Q^\add_{\bullet,0}}\xi_{\bullet,i}\dim S_i+\sum_{i\in Q^\aux_0\setminus Q_{0}}\delta\cdot \dim S_i+(\xi^\delta_{\bullet,\infty}-\xi_{\bullet,\infty})\dim S_\infty\\
    &=\sum_{i\in Q^\add_{\bullet,0}}\xi_{\bullet,i}\dim S_i+\sum_{i\in Q^\aux_0\setminus Q_{0}}\delta\cdot \dim S_i-\sum_{i\in Q^\aux_0\setminus Q_0}\dim S_\infty\cdot\delta\cdot \mathbf v^\aux_i\\
    \text{\tiny (by \eqref{condition on delta} and \eqref{ss implies lower bound})}\quad &>\frac{1}{2}.
\end{align*}
We get a contradiction, so $\overline{\mathfrak{p}}(x,(t_\varepsilon))$ must be $\xi_\bullet$-semistable. This proves \eqref{semi-stable loci inclusion}.

To prove \eqref{stable loci inclusion}. Take an arbitrary point $y\in (\mu^\leg)^{-1}(\mathfrak{Z}^\leg)$ and arbitrary $(t_\varepsilon\in \bC)_{\varepsilon\in \mathcal E^\add}$. Suppose that the $Q^\add_\bullet$-representation $\overline{\mathfrak{p}}(y,(t_\varepsilon))$ is $\xi_\bullet$-stable,~i.e.~$(y,(t_\varepsilon))$ belongs to the left-and-side of \eqref{stable loci inclusion}. Assume $y$ is not $\xi^\delta_\bullet$-stable, then there exists a proper $Q^\aux_{\bullet,0}$-graded subspace $(T_i\subseteq  V_i)_{i\in Q^\aux_{\bullet,0}}$ which is preserved by all linear maps in $y$ and 
\begin{align}\label{not s implies lower bound}
    \sum_{i\in Q^\aux_{\bullet,0}}\xi^\delta_{\bullet,i}\dim T_i\geqslant 0.
\end{align}
The $Q^\add_{\bullet,0}$-graded subspace $(T_i\subseteq  V_i)_{i\in Q^\add_{\bullet,0}}$ is preserved by all maps in $\overline{\mathfrak{p}}(y,(t_\varepsilon))$. There are three possibilities: 1) $T_i= V_i$ for all $i\in Q^\add_{\bullet,0}$; 2) $T_i= 0$ for all $i\in Q^\add_{\bullet,0}$; 3) $(T_i\subseteq  V_i)_{i\in Q^\add_{\bullet,0}}$ is a proper subspace. In case 1), we have
\begin{align*}
    \sum_{i\in Q^\aux_{\bullet,0}}\xi^\delta_{\bullet,i}\dim T_i&=-\sum_{i\in Q^\aux_{\bullet,0}}\xi^\delta_{\bullet,i}\dim V_i/T_i+\sum_{i\in Q^\aux_{\bullet,0}}\xi^\delta_{\bullet,i}\dim V_i\\
    &=-\sum_{i\in Q^\aux_{0}\setminus Q_0}\delta\cdot \dim V_i/T_i<0,
\end{align*}
which contradicts \eqref{not s implies lower bound}. In case 2), there must be a node $j\in Q^\aux_{0}\setminus Q_0$ such that $T_j\neq 0$. Let $i_0$ be the unique node in $Q_{0}$ with the shortest distance to $j$, then the linear map corresponding to the shortest path from $j$ to $i_0$ is injective by the leg-stability. Since $(T_i\subseteq  V_i)_{i\in Q^\aux_{\bullet,0}}$ is preserved by all linear maps in $y$, this implies that $T_{i_0}\neq 0$, so case 2) can not happen either. In case 3), the $\xi_\bullet$-stability of $\overline{\mathfrak{p}}(y,(t_\varepsilon))$ implies that
\begin{align*}
    \sum_{i\in Q^\add_{\bullet,0}}\xi_{\bullet,i}\dim T_i<0.
\end{align*}
Since $\xi_{\bullet,i}\in \bZ$, we must have
\begin{align}\label{s implies upper bound}
    \sum_{i\in Q^\add_{\bullet,0}}\xi_{\bullet,i}\dim T_i\leqslant -1.
\end{align}
Then we get the following upper bound of the left-hand-side of \eqref{not s implies lower bound}:
\begin{align*}
    \sum_{i\in Q^\aux_{\bullet,0}}\xi^\delta_{\bullet,i}\dim T_i&=\sum_{i\in Q^\add_{\bullet,0}}\xi_{\bullet,i}\dim T_i+\sum_{i\in Q^\aux_0\setminus Q_{0}}\delta\cdot \dim T_i-\sum_{i\in Q^\aux_0\setminus Q_0}\dim T_\infty\cdot\delta\cdot \mathbf v^\aux_i\\
    \text{\tiny (by \eqref{condition on delta} and \eqref{s implies upper bound})}\quad &<-\frac{1}{2}.
\end{align*}
We get a contradiction, so $y$ must be $\xi^\delta_\bullet$-semistable. This proves \eqref{stable loci inclusion}.
\end{proof}

Recall that we have fixed a $\theta\in \bQ^{Q_0}$ which is assumed to be generic, i.e. 
$$R(Q,\mathbf v,\mathbf{d})^{\theta\emph{-}ss}=R(Q,\mathbf v,\mathbf{d})^{\theta\emph{-}s}. $$ 
Combine Lemma \ref{compare moment maps} and Lemma \ref{compare stabilities}, we get the following equality
\begin{align}\label{stability compatible}
    \overline{\mathfrak{p}}^{-1}\left(\left((\mu^\add)^{-1}(0)^{\theta \emph{-}s}\right)\right)\cap \left((\mu^\leg)^{-1}(\mathfrak{Z}^\leg)^{\leg \emph{-}s}\times \bC^{\mathcal E^\leg}\right)=(\mu^\aux)^{-1}(\mathfrak{Z}^\aux)^{\theta^\delta \emph{-}s}\times \mathcal L=(\mu^\aux)^{-1}(\mathfrak{Z}^\aux)^{\theta^\delta\emph{-}ss}\times \mathcal L,
\end{align}
for sufficiently small positive $\delta$. Now we combine Lemma \ref{Gro-Spr for quiv as legs} with \eqref{stability compatible}, and get the following.

\begin{Proposition}\label{X tilde as Nakajima quiv}
For sufficiently small positive $\delta$, there exists a commutative diagram of $\sA$-equivariant morphisms:
\begin{equation}\label{eq: X tilde as Nakajima quiv}
\xymatrix{
\widetilde{X} \ar[r]^{\widetilde{\tau}} \ar[d]_{\cong} & \mathcal{H} \ar[d]^{\cong}\\
\widetilde{\cN}_{\theta^\delta}(Q^\aux,\mathbf v^\aux,\mathbf{d}^\aux) \times\mathcal L\ar[r]^{\qquad\quad\,\,\,\,\phi\times\id} & \mathfrak{Z}^\aux\times
\mathcal L,
}
\end{equation}
where $\mathcal L$ is a vector space given in \eqref{equ on vs L}. 
In particular, $\widetilde{X}$ is smooth and $\widetilde{\tau}$ in \eqref{Gro-Spr res for quiv} is a smooth morphism.
\end{Proposition}

\subsection{More properties on diagram \texorpdfstring{\eqref{Gro-Spr res for quiv}}{Gro-Spr}}\label{sec:more on Gro-Spr res for quiv}

Consider the commutative diagram \eqref{Gro-Spr res for quiv}:
\begin{equation*}
\xymatrix{
\widetilde{X} \ar[r]^{\widetilde{\tau}} \ar[d]_{\pi_X} & \mathcal{H} \ar[d]^{p}\\
X \ar[r]^{\tau\,\,} & \mathcal{H}/\Gamma\,,
}
\end{equation*}
and the isomorphism \eqref{eq: X tilde as Nakajima quiv}, let us define a $\Gamma$-invariant open subset in $\mathcal H$ as follows
\begin{align}\label{H circ}
\mathcal H^\circ:=\bigcap_{\gamma\in \Gamma}\gamma\left((\mathfrak{Z}^\aux)^\circ\times \mathcal L\right)
\end{align}
where $(\mathfrak{Z}^\aux)^\circ$ is the open subset in \eqref{open part} for the auxiliary quiver $Q^{\aux}$. Let 
$$X^\circ:=\tau^{-1}(\mathcal H^\circ/\Gamma), \quad \widetilde{X}^\circ:=\widetilde{\tau}^{-1}(\mathcal H^\circ), $$ 
then we have $\widetilde{X}^\circ=\pi_X^{-1}(X^\circ)$.

\begin{Lemma}\label{open part Cartesian}
The following diagram is Cartesian
\begin{equation}\label{Gro-Spr res for quiv_open part}
\xymatrix{
\widetilde{X}^\circ \ar[r]^{\widetilde{\tau}} \ar[d]_{\pi_X} & \mathcal{H}^\circ \ar[d]^{p}\\
X^\circ \ar[r]^{\tau\quad} & \mathcal{H}^\circ/\Gamma\,.
}
\end{equation}
\end{Lemma}

\begin{proof}
Under the isomorphism $\mathfrak{h}^*\cong \mathfrak{Z}^n\times \bC$ in Lemma \ref{T[SU(n)] quiver}, the open subset $(\mathfrak{Z}^n)^\circ\times \bC$ is isomorphic to 
\begin{align*}
    \mathfrak{h}^*{\Bigg\backslash}\bigcup_{\substack{\mathbf u\in \bZ^{n-1}\setminus \{0\}\\ 0\leqslant \mathbf u_i\leqslant i}}  \left(\sum_{i=1}^{n-1}\mathbf u_i\check{\alpha}_i\right)^\perp,
\end{align*}
where $\check{\alpha}_i$ is the $i$-th simple coroot, and for $\lambda\in \mathfrak{h}$, $\lambda^\perp$ is the hyperplane $\{v\in \mathfrak{h}^*\colon \lambda(v)=0\}$. In particular, $(\mathfrak{Z}^n)^\circ\times \bC$ is contained in the complement of the union of coroot hyperplanes: $$(\mathfrak{Z}^n)^\circ\times \bC\subset \mathring{\mathfrak{h}}^*:=\mathfrak{h}^*{\Big\backslash}\bigcup_{\check{\alpha}\in\text{coroots}}  \check{\alpha}^\perp.$$
It follows that $\mathcal H^\circ\subset \prod_{\varepsilon\in \mathcal E^{\add}}\mathring{\mathfrak{h}}^*_{\varepsilon}$, where $\mathring{\mathfrak{h}}^*_{\varepsilon}$ is the complement of the union of coroot hyperplanes in the dual of the Cartan subalgebra of $\mathfrak{gl}(\bv_{h(\varepsilon)})$.

It is well-known that the commutative diagram \eqref{Grothendieck-Springer resolution} becomes Cartesian after base change to $\mathring{\mathfrak{h}}^*/W$, precisely, the following diagram is Cartesian:
\begin{equation*}
\xymatrix{
\nu^{-1}(\mathring{\mathfrak{h}}^*) \ar[r]^-{\nu_{\mathfrak{l}}} \ar[d]_{\pi_{\mathfrak{l}}} & \mathring{\mathfrak{h}}^* \ar[d]^{p_{\mathfrak{l}}}\\
\rho^{-1}(\mathring{\mathfrak{h}}^*/W) \ar[r]^-{\rho_{\mathfrak{l}}} & \mathring{\mathfrak{h}}^*/W.
}
\end{equation*}
Therefore, the commutative diagram \eqref{Gro-Spr res for quiv before quot} becomes Cartesian after base change to $\cH^\circ/\Gamma$, that is, the following diagram is Cartesian:
\begin{equation*}
\xymatrix{
\widetilde{R}(Q^{\add},\mathbf v,\mathbf{d})^\circ \ar[r] \ar[d]_{\pi_R}  & \cH^\circ \ar[d]^{p}\\
{R}(Q^{\add},\mathbf v,\mathbf{d})^\circ \ar[r]  & \cH^\circ/\Gamma, }
\end{equation*}
where ${R}(Q^{\add},\mathbf v,\mathbf{d})^\circ={R}(Q^{\rem},\mathbf v,\mathbf{d})\times \rho^{-1}(\cH^\circ/\Gamma)$ and $\widetilde{R}(Q^{\add},\mathbf v,\mathbf{d})^\circ={R}(Q^{\rem},\mathbf v,\mathbf{d})\times \nu^{-1}(\cH^\circ)$. After taking Hamiltonian reduction, the Cartesian diagram \eqref{Gro-Spr res for quiv_open part} follows.
\end{proof}

\begin{Lemma}\label{X^circ is affine}
$X^\circ$ is an affine variety.
\end{Lemma}

\begin{proof}
Note that $\widetilde{\tau}\colon \widetilde{X}^\circ\to \mathcal H^\circ$ is an affine morphism as discussed in Example \ref{extended nak quiv}, in particular $\widetilde{X}^\circ$ is an affine variety. By Lemma \ref{open part Cartesian}, $\pi_X\colon \widetilde{X}^\circ\to X^\circ$ is finite and surjective. Then $X^\circ$ is an affine variety by Chevalley's theorem on affineness \cite[Thm.~12.39]{GW}.
\end{proof}

\begin{Proposition}\label{prop of tau}
In the commutative diagram \eqref{Gro-Spr res for quiv}, we have:
\begin{enumerate}
\setlength{\parskip}{1ex}
    \item The morphism $\tau\colon X\to \calH/\Gamma$ is flat.
    \item Let $h\in \calH$ be an arbitrary point, then $\widetilde{X}_h:=\widetilde{\tau}^{-1}(h)$ is smooth and connected, $X_{p(h)}:=\tau^{-1}(p(h))$ is irreducible and normal, and $\pi_X\colon \widetilde{X}_h\to X_{p(h)}$ is a resolution of singularities.
\end{enumerate}
\end{Proposition}


\begin{proof}
Let us denote the compositions of the horizontal maps in the commutative diagram \eqref{Gro-Spr res for quiv before quot} to be $\overline{\nu}$ and $\overline{\rho}$. By Theorem \ref{GS resolution}(5), the morphism $\pi_R\colon \overline{\nu}^{-1}(h)\to \overline{\rho}^{-1}(p(h))$ is proper and surjective, and there exists an open subscheme $U\subseteq  \overline{\rho}^{-1}(p(h))$ such that $\pi_R^{-1}(U)$ is the largest open subscheme in $\overline{\nu}^{-1}(h)$ on which $\pi_R$ is quasi-finite. In fact, $\pi_R$ induces isomorphism between $\pi_R^{-1}(U)$ and $U$. It follows that the morphism 
\begin{align*}
    \overline{\nu}^{-1}(h)\cap \pi_R^{-1}((\mu^\add)^{-1}(0)^{\theta\emph{-}s})\to \overline{\rho}^{-1}(p(h))\cap  (\mu^\add)^{-1}(0)^{\theta\emph{-}s}
\end{align*}
is a proper and surjective, and $\pi_R^{-1}(U\cap (\mu^\add)^{-1}(0)^{\theta\emph{-}s})$ is the largest open subscheme of $\overline{\nu}^{-1}(h)\cap \pi_R^{-1}((\mu^\add)^{-1}(0)^{\theta\emph{-}s})$ on which the morphism is quasi-finite. Moreover $\pi_R$ induces isomorphism between $\pi_R^{-1}(U\cap (\mu^\add)^{-1}(0)^{\theta\emph{-}s})$ and $U\cap (\mu^\add)^{-1}(0)^{\theta\emph{-}s}$. It worths mention that $U\cap (\mu^\add)^{-1}(0)^{\theta\emph{-}s}$ might be empty, but we shall see shortly that this can not happen.

Taking quotient by $G$, we get the following commutative diagram
\begin{equation*}
\xymatrix{
\overline{\nu}^{-1}(h)\cap \pi_R^{-1}((\mu^\add)^{-1}(0)^{\theta\emph{-}s})\ar[r]^{\pi_R} \ar[d]_{\cdot/\!\!/ G} & \overline{\rho}^{-1}(p(h))\cap  (\mu^\add)^{-1}(0)^{\theta\emph{-}s} \ar[d]^{\cdot/\!\!/ G}\\
\widetilde{X}_h \ar[r]^{\pi_X} & X_{p(h)}\,,
}
\end{equation*}
which is Cartesian because $G$-action on $\theta$-stable locus is free. In particular $\pi_X\colon \widetilde{X}_h\to X_{p(h)}$ is proper and surjective. By the maximality of $U$, $U$ is $G$-invariant, and let us denote $V:=\left(U\cap (\mu^\add)^{-1}(0)^{\theta\emph{-}s}\right)/\!\!/ G$. Then $\pi_X^{-1}(V)$ is the largest open subscheme in $\widetilde{X}_h$ on which $\pi_X$ is quasi-finite, and $\pi_X$ induces isomorphism between $\pi_X^{-1}(V)$ and $V$. 

We note that $\widetilde{X}_h$ is isomorphic to the Nakajima quiver variety 
\begin{align*}
\cN_{\theta^\delta,\bar{h}}(Q^\aux,\mathbf v^\aux,\mathbf{d}^\aux)=(\mu^{\aux})^{-1}(\bar{h})^{\theta^\delta\emph{-}ss}/\!\!/G^{\aux},
\end{align*}
where $\bar{h}$ is the image of $h$ under the projection of vector space $\cH\to \mathfrak{Z}^{\aux}$. In particular $\widetilde{X}_h$ is irreducible\,\footnote{Because Nakajima quiver varieties are smooth and connected \cite[\S 1]{CB}.}. This implies that $X_{p(h)}$ is irreducible and $\dim X_{p(h)}\leqslant \dim \widetilde{X}_h$. By \cite[Cor.~14.95]{GW} we have
\begin{align*}
    \dim X_{p(h)}\geqslant \dim X-\dim \calH/\Gamma.
\end{align*}
By the smoothness of $\widetilde{\tau}$, we have
\begin{align*}
    \dim \widetilde{X}_h= \dim \widetilde{X}-\dim \calH.
\end{align*}
Since $\pi_X$ is generically finite, we have $\dim X=\dim \widetilde{X}$, hence we get another inequality $\dim X_{p(h)}\geqslant \dim \widetilde{X}_h$. This forces the equality $\dim X_{p(h)}= \dim \widetilde{X}_h$ to hold. In particular $\pi_X\colon \widetilde{X}_h\to X_{p(h)}$ is generically finite. This implies that $V$ is nonempty. In particular $\pi_X\colon \widetilde{X}_h\to X_{p(h),\mathrm{red}}$ is a proper birational morphism.

By the above argument, $\dim X_{p(h)}= \dim {X}-\dim \calH/\Gamma$. We note that $\calH/\Gamma$ is smooth because it is a product of $\bC^n/\mathfrak{S}_n$ for various of $n$ and each $\bC^n/\mathfrak{S}_n$ is smooth. Then $\tau$ is flat by \cite[Thm.~14.126]{GW}. This proves (1).

To prove (2), it remains to show that $X_{p(h)}$ is a normal scheme. By the flatness of $\tau\colon X\to \calH/\Gamma$ and the smoothness of $X$ and $\calH/\Gamma$, $X_{p(h)}$ is an l.c.i.~scheme, in particular $X_{p(h)}$ is Cohen-Macaulay. Since $\pi_X$ induces isomorphism $\pi_X^{-1}(V)\cong V$ and $\pi_X^{-1}(V)$ is smooth, $X_{p(h)}$ is generically smooth. 

We claim that $X_{p(h)}\setminus V\subseteq  X_{p(h)}$ has codimension $>1$. Suppose the claim is false, let us take $D$ to be a codimension one irreducible component of $X_{p(h)}\setminus V$, then the dimension of fibers of $\pi_X$ along $D$ is positive, hence 
$$\dim \pi_X^{-1}(D)\geqslant \dim D+1=\dim \widetilde{X}_h.$$ 
This forces $\widetilde{X}_h=\pi_X^{-1}(D)$ by the connectedness of $\widetilde{X}_h$, contradicting the fact that $\widetilde{X}_h$ maps surjectively onto $X_{p(h)}$. This proves our claim. Therefore $X_{p(h)}$ is a normal scheme by Serre's criterion of normality \cite[\href{https://stacks.math.columbia.edu/tag/0345}{Lem.~0345}]{stacks-project}. 
\end{proof}

\begin{Remark}\label{pi_X is a symp res}
$X$ has a natural Poisson structure that is induced from the Poisson structure on $R(Q^{\add},\mathbf v,\mathbf{d})$ constructed in the beginning of \S\ref{sec sym quiv vs nak quiv}. Similarly, $\widetilde{X}$ has a natural Poisson structure that is induced from the Poisson structure on $R(Q^{\aux},\mathbf v^\aux,\mathbf{d}^\aux)$. The morphism $\mathfrak{p}\colon R(Q^{\aux},\mathbf v^{\aux},\mathbf{d}^{\aux})\times \bC^{\mathcal E^\add}\to R(Q^{\add},\mathbf v,\mathbf{d})$ in Definition \ref{map A_n leg to loop} is Poisson, so $\pi_X\colon \widetilde{X}\to X$ is also a Poisson morphism. Moreover, the morphisms $\tau\colon X\to \calH/\Gamma$ and $\widetilde{\tau}\colon \widetilde{X}\to \calH$ are Poisson where the targets are endowed with a trivial Poisson structure. It follows that $\pi_X\colon \widetilde{X}_h\to X_{p(h)}$ is a Poisson morphism. We note that the induced Poisson structure on $\widetilde{X}_h$ is exactly the natural symplectic structure on the Nakajima quiver variety; thus $\pi_X\colon \widetilde{X}_h\to X_{p(h)}$ is a symplectic resolution.
\end{Remark}

The restriction of $\tau$ to a torus fixed component is also flat. 

\begin{Proposition}\label{tau at fixed loci is flat}
Suppose that $\sigma$ is a cocharacter of $\sA$, and $F\in \Fix_\sigma(X)$ is a $\sigma$-fixed locus. Then the morphism $\tau|_F\colon F\to \calH/\Gamma$ is flat. Moreover, $(\tau|_F)^{-1}(0)$ is irreducible.
\end{Proposition}

\begin{proof}
$F$ determines a lift $\xi\colon \bC^*\to G\times \sA$ of $\sigma\colon \bC^*\to \sA$ along the projection $G\times \sA\to \sA$. Then 
\begin{align}\label{fixed comp as quot}
    F=\left((\mu^\add)^{-1}(0)^{\theta\emph{-}s}\cap R(Q^\add,\mathbf v,\mathbf{d})^\xi\right)/\!\!/ G^{\xi}.
\end{align}
Let us take $\widetilde{F}\in \Fix_\sigma(\widetilde{X})$ such that $\pi_X(\widetilde{F})\subseteq  F$. $\widetilde{F}$ determines a further lift $\widetilde{\xi}\colon \bC^*\to G^\aux\times \sA$  of $\xi\colon \bC^*\to G\times \sA$ along the projection $G^\aux\times \sA\to G\times \sA$. In view of the equality \eqref{stability compatible}, we get an isomorphism 
\begin{align*}
    \widetilde{F}\cong \left(\overline{\mathfrak{p}}^{-1}((\mu^\add)^{-1}(0)^{\theta\emph{-}s})\cap \left((\mu^\leg)^{-1}(\mathfrak{Z}^\leg)^{\leg\emph{-}s}\times \bC^{\mathcal E^\leg}\right)\cap R(Q^\aux,\mathbf v^\aux,\mathbf{d}^\aux)^{\widetilde{\xi}}\right)/\!\!/ (G^\aux)^{\widetilde{\xi}}.
\end{align*}
Decomposing $G^\aux\times \sA$ into $G\times \sA\times G^\leg$, we can write $\widetilde{\xi}=(\xi,\xi^\leg)$ where $\xi^\leg\colon \bC^*\to G^\leg$ is the corresponding component of $\widetilde{\xi}$. Then we have $(G^\aux)^{\widetilde{\xi}}=G^\xi\times (G^\leg)^{\xi^\leg}$. In view of the following isomorphism
\begin{align}\label{tilde fixed comp as quot}
    \widetilde{R}(Q^\add,\mathbf v,\mathbf{d})\cong (\mu^\leg)^{-1}(\mathfrak{Z}^\leg)^{\leg\emph{-}s}/\!\!/ G^\leg\times \bC^{\mathcal E^\leg}
\end{align}
in Lemma \ref{Gro-Spr for quiv as legs}, the quotient $\left(\left((\mu^\leg)^{-1}(\mathfrak{Z}^\leg)^{\leg\emph{-}s}\times \bC^{\mathcal E^\leg}\right)\cap R(Q^\aux,\mathbf v^\aux,\mathbf{d}^\aux)^{\widetilde{\xi}}\right)/\!\!/ (G^\leg)^{\xi^\leg}$ is identified with a $\xi$-fixed component of $\widetilde{R}(Q^\add,\mathbf v,\mathbf{d})$. Let us denote this component by $M$. Then we have
\begin{align*}
    \widetilde{F}\cong \left(\pi_R^{-1}((\mu^\add)^{-1}(0)^{\theta\emph{-}s})\cap M\right)/\!\!/ G^\xi.
\end{align*}

For a complex reductive group $L$ together with a cocharacter $\lambda\colon \bC^*\to L$, consider the Grothendieck-Springer morphism $\pi\colon \widetilde{\mathfrak{l}}^*\to \mathfrak{l}^*$ in the diagram \eqref{Grothendieck-Springer resolution}.
It is known that $\pi$ maps every connected component of $(\widetilde{\mathfrak{l}}^*)^\lambda$ surjectively onto $(\mathfrak{l}^*)^{\lambda}$ \footnote{Notice that restriction of $\pi$ to the preimage of regular semisimple locus $\pi^{-1}((\mathfrak{l}^*)^{\mathrm{rs}})\to (\mathfrak{l}^*)^{\mathrm{rs}}$ is a finite \'etale covering, so every connected component of $\pi^{-1}((\mathfrak{l}^*)^{\mathrm{rs}})^\lambda$ is a connected component of $\pi^{-1}((\mathfrak{l}^*)^{\mathrm{rs}}\cap (\mathfrak{l}^*)^\lambda)$. In particular, every connected component of $\pi^{-1}((\mathfrak{l}^*)^{\mathrm{rs}})^\lambda$ maps surjective onto $(\mathfrak{l}^*)^{\mathrm{rs}}\cap (\mathfrak{l}^*)^\lambda$. Note that the restriction of the morphism $\nu$ in the diagram \eqref{Grothendieck-Springer resolution} to $(\widetilde{\mathfrak{l}}^*)^\lambda$ is smooth by the same argument as the proof of Lemma \ref{symp res has cond star}; therefore taking intersection with open subset $\pi^{-1}((\mathfrak{l}^*)^{\mathrm{rs}})=\nu^{-1}((\mathfrak{h}^*)^{\mathrm{rs}})$ gives a one-to-one correspondence between $\Fix_\lambda(\widetilde{\mathfrak{l}}^*)$ and $\Fix_\lambda(\pi^{-1}((\mathfrak{l}^*)^{\mathrm{rs}}))$. This implies that every connected component of $(\widetilde{\mathfrak{l}}^*)^\lambda$ maps surjective onto $(\mathfrak{l}^*)^\lambda$.}. Applying the aforementioned fact to $L=\prod_{\varepsilon\in \mathcal E^\add}\GL(\mathbf v_{h(\varepsilon)})$ and $\lambda$ being the cocharacter induced by $\bC^*\xrightarrow{\xi}G\hookrightarrow L$, we see that $\pi_R$ maps $M$ surjectively onto ${R}(Q^\add,\mathbf v,\mathbf{d})^\xi$. Comparing \eqref{fixed comp as quot} and \eqref{tilde fixed comp as quot}, we see that $\pi_X$ maps $\widetilde{F}$ surjectively onto $F$. Note that $\widetilde{\tau}|_{\widetilde{F}}\colon \widetilde{F}\to \calH$ is smooth by the argument of proof of Lemma \ref{symp res has cond star}, so $\widetilde{F}\cap \widetilde{X}^\circ$ is nonempty, in particular $\pi|_{\widetilde{F}}\colon \widetilde{F}\to X$ is generically finite. Thus the morphism $\widetilde{F}\to F$ is proper, surjective, and generically finite. In particular $\dim \widetilde{F}=\dim F$.

For $h\in\cH$, define $\widetilde{F}_h:=\widetilde{F}\times_\cH \{h\}$; and for $b\in \cH/\Gamma$, define $F_b:=F\times_{\cH/\Gamma}\{b\}$. By the smoothness of $\widetilde{\tau}|_{\widetilde{F}}\colon \widetilde{F}\to \calH$, we have
\begin{align*}
    \dim \widetilde{F}_h=\dim \widetilde{F}-\dim \calH
\end{align*}
for any $h\in \calH$. Since $\widetilde{F}\to F$ is surjective and $p\colon \calH\to \calH/\Gamma$ is finite, we have 
\begin{align*}
    \dim F_b\leqslant \max_{h\in p^{-1}(b)}\dim \widetilde{F}_h=\dim \widetilde{F}-\dim \calH=\dim {F}-\dim \calH/\Gamma
\end{align*}
for any $b\in \calH/\Gamma$. By \cite[Cor.~14.95]{GW} we have
\begin{align*}
    \dim F_b\geqslant \dim F-\dim \calH/\Gamma.
\end{align*}
This forces the equality $\dim F_b= \dim F-\dim \calH/\Gamma$ to hold. Then $\tau|_F$ is flat by \cite[Thm.~14.126]{GW}. 

Finally, $\widetilde{F}_0$ maps surjectively onto $F_0$, because $0\in \calH$ is the unique preimage of $0\in \calH/\Gamma$ under the morphism $p$. Since $\widetilde{F}_0$ is isomorphic to a $\sigma$-fixed component of Nakajima quiver variety $\cN_{\theta^\delta}(Q^\aux,\mathbf v^\aux,\mathbf{d}^\aux)$, $\widetilde{F}_0$ itself is a Nakajima quiver variety. Then $\widetilde{F}_0$ is irreducible by \cite[\S 1]{CB}. Thus $F_0$ is irreducible.
\end{proof}

\subsection{Smallness of the affinization morphism}\label{subsec smallness}
The Jordan-H\"older morphism $$\mathsf{JH}\colon X=\cM_\theta(Q,\mathbf v,\mathbf{d})\to \cM_0(Q,\mathbf v,\mathbf{d})=Y$$ maps $\theta$-stable quiver representation to its semisimplification, where $\cM_0(Q,\mathbf v,\mathbf{d})=R(Q,\bv,\bd)/\!\!/G$ is the affine quotient ($G=\prod_{i\in Q_0}\GL(\bv_i)$). Geometrically, $\mathsf{JH}$ is induced from the open immersion $R(Q,\bv,\bd)^{\theta\emph{-}s}\hookrightarrow R(Q,\bv,\bd)$ by taking $G$ quotients on both sides.
\begin{Lemma}\label{lem proper birational}
$\mathsf{JH}$ is proper and birational. Moreover, $\bC[Y]=\Gamma(X,\mathcal O_X)$, i.e. $Y$ is the affinization of $X$. 
\end{Lemma}

\begin{proof}
By geometric invariant theory, $\mathsf{JH}$ is proper. Applying the Quantization Theorem \cite[Thm.~3.29]{HL} to the case $F^\bullet=G^\bullet=$ structure sheaf of the stack $[R(Q,\mathbf v,\mathbf{d})/G]$, we see that the natural map 
\begin{align}\label{aff isom}
    \bC[Y]=H^0([R(Q,\mathbf v,\mathbf{d})/G],\mathcal O_{[R(Q,\mathbf v,\mathbf{d})/G]})\to H^0(R(Q,\mathbf v,\mathbf{d})^{\theta\emph{-}s}/G,\mathcal O_{[R(Q,\mathbf v,\mathbf{d})/G]})=\Gamma(X,\mathcal O_X)
\end{align}
is an isomorphism. Since $Y$ is affine, \eqref{aff isom} implies that the natural map $\mathcal O_Y\to \mathsf{JH}_*\mathcal O_X$ is an isomorphism.

The map $R(Q,\bv,\bd)\subset R(Q^{\add},\bv,\bd)\to \cH/\Gamma$ induced from the bottom of \eqref{Gro-Spr res for quiv before quot} is $G$-invariant, and it descends to a morphism $\bar{\tau}\colon Y=\cM_0(Q,\mathbf v,\mathbf{d})\to \cH/\Gamma$. We have $\tau=\bar{\tau}\circ\mathsf{JH}$ by construction. Recall the open subset $\cH^\circ$ defined in \eqref{H circ}, and define $Y^\circ :=\bar{\tau}^{-1}(\cH^\circ/\Gamma)$. Then $\mathsf{JH}\colon X^\circ\to Y^\circ$ is proper and affine, and therefore finite. Then it follows from the isomorphism $\mathcal O_Y\cong \mathsf{JH}_*\mathcal O_X$ that $\mathsf{JH}\colon X^\circ\to Y^\circ$ is an isomorphism. In particular, $\mathsf{JH}$ is birational.
\end{proof}

The main result of this subsection is the following.
\begin{Proposition}\label{prop smallness}
All irreducible components of $X\times_Y X$ except the diagonal $\Delta_X$ have dimension strictly smaller than $\dim X$. In particular, $\dim X\times_Y X=\dim X$, and $\mathsf{JH}\colon X\to Y$ is small.
\end{Proposition}

\begin{proof}
As we have seen in the proof of Lemma \ref{lem proper birational}, $\mathsf{JH}$ induces isomorphism $X^\circ\cong Y^\circ$; therefore 
$$X\times_Y X\times_{\cH/\Gamma} \cH^\circ/\Gamma\cong X^\circ\times_{Y^\circ} X^\circ= \Delta_{X^\circ}. $$ 
On the other hand, we claim that for any $b\in \cH/\Gamma$, the dimension of the fiber $X\times_Y X\times_{\cH/\Gamma} \{b\}$ is equal to $\dim X_b$, where $X_b:=\tau^{-1}(b)$. By the flatness of $\tau$ (Proposition \ref{prop of tau}), it is enough to show 
$$\dim X\times_Y X\times_{\cH/\Gamma} \{b\}\leqslant \dim X_0. $$ 
Consider the $\bC^*$ action on $R(Q,\mathbf v,\mathbf{d})$ that scales the whole vector space with weight $1$. This $\bC^*$ action commutes with $G$-action, and therefore it induces a $\bC^*$ action on the GIT quotient $X$ as well as on the affine quotient $Y$. Note that $\bC^*$ action on $Y$ contracts the latter to the unique fixed point $\mathfrak{o}$ which is the image of $\{0\}\in R(Q,\mathbf v,\mathbf{d})$. By the properness of $\mathsf{JH}\colon X\to Y$, $\forall\,\, x\in X$, $\lim_{t\to 0} t\cdot x$ exists and belongs to $\mathsf{JH}^{-1}(\mathfrak{o})$. Moreover $\bC^*$ acts on $\mathcal H/\Gamma$ with positive weights, such that $\tau\colon X\to \mathcal H/\Gamma$ is $\bC^*$-equivariant. This implies that $\lim_{t\to 0} t\cdot x\in X_0$. 

Take an arbitrary irreducible component $W$ of $X\times_Y X\times_{\cH/\Gamma} \{b\}$, let $W'$ be the closure or $\bC^*$ orbits of $W$ in $X\times_Y X$, i.e. $W'=\overline{\bC^*\cdot W}$. Then $W'_0:=W'\times_{\cH/\Gamma}\{0\}$ is nonempty by the above argument, and $\dim W'_0\geqslant\dim W$ as the former is a special fiber and the latter is the generic fiber. Since $Y_0:=\bar{\tau}^{-1}(0)$ is the affine Nakajima quiver variety $\cN_{0}(Q^\aux,\mathbf v^\aux,\mathbf{d}^\aux)=(\mu^{\aux})^{-1}(0)/\!\!/G^{\aux}$ and the composition $\mathsf{JH}\circ \pi_X\colon \widetilde{X}_0\to Y_0$ is identified with the Jordan-H\"older map of Nakajima quiver varieties:
\begin{align*}
\cN_{\theta^{\delta}}(Q^\aux,\mathbf v^\aux,\mathbf{d}^\aux)\to \cN_{0}(Q^\aux,\mathbf v^\aux,\mathbf{d}^\aux),
\end{align*}
we see that $\widetilde{X}_0\to Y_0$ is a symplectic resolution, which is semismall \cite{Kal}. Then the claim follows from
\begin{align*}
\dim W\leqslant\dim W'_0\leqslant \dim X_0\times_{Y_0} X_0\leqslant \dim \widetilde{X}_0\times_{Y_0} \widetilde{X}_0=\dim \widetilde{X}_0=\dim X_0\,.
\end{align*}
Let $Z\subset X\times_Y X$ be an irreducible component which is different from the diagonal $\Delta_X$. Suppose $Z^\circ:=Z\times_{\cH/\Gamma}\cH^\circ/\Gamma$ is nonempty, then $Z$ is the closure of $Z^\circ$. Since $X^\circ\times_{Y^\circ} X^\circ= \Delta_{X^\circ}$, $Z^\circ$ is contained in $\Delta_{X^\circ}$. This implies that $Z\subseteq \Delta_X$, and contradicts with the choice of $Z$. Thus $Z^\circ=\emptyset$. Then it follows from the above claim that $\dim Z\leqslant \dim X_b+\dim (\cH\setminus\cH^\circ)/\Gamma=\dim X-1$.
\end{proof}

The restriction of $\mathsf{JH}$ to a torus fixed component is also a birational and small morphism onto its image.

\begin{Proposition}\label{prop smallness fixed locus}
Suppose that $\sigma$ is a cocharacter of $\sA$, and $F\in \Fix_\sigma(X)$. Then the morphism $$\mathsf{JH}|_F\colon F\to Y$$ is birational onto its image. Moreover, all irreducible components of $F\times_Y F$ except the diagonal $\Delta_F$ have dimension strictly smaller than $\dim F$. In particular, $\dim F\times_Y F=\dim F$, and $\mathsf{JH}\colon F\to Y$ is small onto its image.
\end{Proposition}

\begin{proof}
By the flatness of $\tau|_F\colon F\to \calH/\Gamma$ (Proposition \ref{tau at fixed loci is flat}), $F^\circ:=F\times_{\calH/\Gamma} \calH^{\circ}/\Gamma$ is nonempty. Since $\mathsf{JH}\colon X^\circ\to Y^\circ$ is an isomorphism, $F^\circ$ is mapped isomorphically onto its image in $Y^\circ$. In particular, $\mathsf{JH}|_F\colon F\to Y$ is birational onto its image. 

Let $K\subset F\times_Y F$ be an irreducible component which is different from the diagonal $\Delta_F$. The same argument as in the proof of Proposition \ref{prop smallness} shows that $K\times_{\cH/\Gamma}\cH^{\circ}/\Gamma$ is empty. Then we have
\begin{align*}
\dim K&\leqslant\dim \cH/\Gamma-1+\sup_{b\in \cH/\Gamma}\dim F\times_Y F\times_{\cH/\Gamma}\{b\}\\
{\text{\scriptsize by \eqref{F1xF2 fiber dim bound}}}\;&\leqslant \dim \cH/\Gamma-1+\frac{1}{2}(\dim F+\dim F)-\dim \cH/\Gamma\\
&=\dim F-1.\qedhere  
\end{align*}
\end{proof}

\subsection{Some dimension estimates}

Next we give several estimates on the dimensions of intersections of torus fixed loci of symmetric quiver varieties. They will be used in \cite{COZZ2}.
\begin{Lemma}\label{lem F1xF2 fiber}
Suppose that $\sigma$ is a cocharacter of $\sA$, and $F_1,F_2\in \Fix_\sigma(X)$. Then for any $b\in \cH/\Gamma$,
\begin{align}\label{F1xF2 fiber dim bound}
\dim F_1\times_Y F_2\times_{\cH/\Gamma}\{b\}\leqslant \frac{1}{2}(\dim F_1+\dim F_2)-\dim \cH/\Gamma.
\end{align}
\end{Lemma}

\begin{proof}
By the $\bC^*$ contraction argument in the proof of Proposition \ref{prop smallness}, 
$$\dim F_1\times_Y F_2\times_{\cH/\Gamma}\{b\}\leqslant \dim F_1\times_Y F_2\times_{\cH/\Gamma}\{0\}, $$ 
so it is enough to prove \eqref{F1xF2 fiber dim bound} for $b=0$.

For $i=1,2$, let $\widetilde{F}_i$ be the $\sigma$-fixed component of $\widetilde{X}$ that dominates $F_i$ (see the proof of Proposition \ref{tau at fixed loci is flat}). Since $\widetilde{F}_i\to F_i$ is generically finite, we have $\dim \widetilde{F}_i=\dim F_i$. Define $F_{i,0}:=F_i\times_{\cH/\Gamma}\{0\}$ and $\widetilde{F}_{i,0}:=\widetilde{F}_i\times_{\cH/\Gamma}\{0\}$ for $i=1,2$. There is a naturally induced proper surjective morphism $\widetilde{F}_{1,0}\times_{Y_0} \widetilde{F}_{2,0}\to {F}_{1,0}\times_{Y_0} {F}_{2,0}$, so $$\dim {F}_{1,0}\times_{Y_0} {F}_{2,0}\leqslant \dim \widetilde{F}_{1,0}\times_{Y_0} \widetilde{F}_{2,0}. $$ 
Note that $\sigma$ preserves the symplectic form $\Omega$ on the Nakajima quiver variety $\widetilde{X}_0=\cN_{\theta^\delta}(Q^\aux,\mathbf v^\aux,\mathbf{d}^\aux)$, so the restriction of $\Omega$ to $\sigma$-fixed components $\widetilde{F}_{1,0}$ and $\widetilde{F}_{2,0}$ are nondegenerate. Let us equip $\widetilde{F}_{1,0}\times \widetilde{F}_{2,0}$ with the symplectic structure $\Omega_{12}:=\Omega|_{\widetilde{F}_{1,0}}\boxminus \Omega|_{\widetilde{F}_{2,0}}$. Then we claim that the subvariety $\widetilde{F}_{1,0}\times_{Y_0} \widetilde{F}_{2,0}\subset \widetilde{F}_{1,0}\times \widetilde{F}_{2,0}$ is isotropic with respect to $\Omega_{12}$. As we have seen in the proof of Proposition \ref{prop smallness}, $\mathsf{JH}\circ\pi_X\colon \widetilde{X}_0\to Y_0$ is a symplectic resolution. Let $W$ be an arbitrary irreducible component of $\widetilde{F}_{1,0}\times_{Y_0} \widetilde{F}_{2,0}$ with reduced scheme structure. Let $\eta_1$ be the composition $W\to \widetilde{F}_{1,0}\to \widetilde{X}_0$, and let $\eta_2$ be the composition $W\to \widetilde{F}_{2,0}\to \widetilde{X}_0$, then $$\mathsf{JH}\circ\pi_X\circ \eta_1=\mathsf{JH}\circ\pi_X\circ \eta_2.$$ 
Denote $V$ to be the image of $q\colon W\to Y_0$ where $q=\mathsf{JH}\circ\pi_X\circ \eta_1$, and we shall still denote the morphism $W\to V$ by $q$. By \cite[Lem.~2.9]{Kal}, we can replace $V$ and $W$ by open subvarieties $U$ and $Z$ respectively, such that $U$ and $Z$ are smooth and $q\colon Z\to U$ is smooth, and $\eta_1^*\Omega=q^*\Omega_{U,1}$ and $\eta_2^*\Omega=q^*\Omega_{U,2}$, for some $\Omega_{U,1},\Omega_{U,2}\in\Omega^2(U)$. By the construction in the proof of \cite[Lem.~2.9]{Kal}, for $i=1,2$, $\Omega_{U,i}$ only depends on the morphism $U\to Y_0$ and does not depend on $\eta_i$, so we have $\Omega_{U,1}=\Omega_{U,2}$. It follows that the restriction of $\Omega_{12}$ to $Z$ equals to $\eta_1^*\Omega-\eta_2^*\Omega=0$. This implies that $W$ is isotropic, and therefore $\widetilde{F}_{1,0}\times_{Y_0} \widetilde{F}_{2,0}$ is isotropic.

By the isotropic property, $\dim \widetilde{F}_{1,0}\times_{Y_0} \widetilde{F}_{2,0}\leqslant\frac{1}{2} \dim\widetilde{F}_{1,0}\times \widetilde{F}_{2,0}$. Thus,
\begin{align*}
\dim {F}_{1,0}\times_{Y_0} {F}_{2,0}\leqslant \dim \widetilde{F}_{1,0}\times_{Y_0} \widetilde{F}_{2,0}&\leqslant\frac{1}{2} \dim\widetilde{F}_{1,0}\times \widetilde{F}_{2,0}\\
{\text{\scriptsize (by the smoothness of $\widetilde{F}_{i,0}\to \cH$)}}\;&=\frac{1}{2} (\dim\widetilde{F}_{1}+\dim \widetilde{F}_{2})-\dim \cH\\
&=\frac{1}{2} (\dim{F}_{1}+\dim {F}_{2})-\dim \cH/\Gamma.\qedhere  
\end{align*}
\end{proof}

\begin{Proposition}\label{prop F1xF2 cap Z}
Suppose that $\sigma$ is a cocharacter of $\sA$, and $F_1,F_2\in \Fix_\sigma(X)$. Let $Z\subset X\times_Y X$ be an irreducible component which is different from the diagonal $\Delta_X$. Then 
\begin{align}
\dim (F_1\times F_2)\cap Z\leqslant \frac{1}{2}(\dim F_1+\dim F_2)-1.
\end{align}
\end{Proposition}

\begin{proof}
Let $T$ be an irreducible component of $(F_1\times F_2)\cap Z$ with maximal dimension and denote $T_b:=T\times_{\cH/\Gamma}\{b\}$ for any $b\in \cH/\Gamma$. Since $T_b\subseteq F_1\times_Y F_2\times_{\cH/\Gamma}\{b\}$, we have $\dim T_b\leqslant \frac{1}{2}(\dim F_1+\dim F_2)-\dim \cH/\Gamma$ by Lemma \ref{lem F1xF2 fiber}. In the proof of Proposition \ref{prop smallness} we have shown that $Z\times_{\cH/\Gamma}\cH^\circ/\Gamma$ is empty, so
\begin{equation*}
\dim (F_1\times F_2)\cap Z\leqslant \dim \cH/\Gamma-1+\sup_{b\in \cH/\Gamma} \dim T_b\leqslant \frac{1}{2}(\dim F_1+\dim F_2)-1.\qedhere  
\end{equation*}
\end{proof}

\begin{Remark}\label{rmk F1xF2 distinct}
If $F_1\neq F_2$, then $(F_1\times F_2)\cap \Delta_X=\emptyset$, and it follows from Proposition \ref{prop F1xF2 cap Z} that
\begin{align}
    \dim F_1\times_Y F_2\leqslant \frac{1}{2}(\dim F_1+\dim F_2)-1.
\end{align}
\end{Remark}

\section{Stable envelope correspondences for symmetric quiver varieties}
\subsection{Definitions}
We say that a cocharacter $\sigma \colon \bC^* \to \sA$ is \emph{generic} if $X^{\sigma} = X^\sA$. Let us recall the wall-and-chamber structure on $\Lie(\sA)_\bR$ described in \cite{COZZ1}, such that $\sigma$ is generic if and only if it lies in a chamber. 

The \textit{torus roots} are the set of $\sA$-weights $\{\alpha\}$ occurring in the normal bundle to $X^\sA$. 
A connected component of the complement of union of (finite) root hyperplanes is called a \textit{chamber}, i.e.
\begin{align*}
    \Lie(\sA)_\bR\bigg{\backslash} \bigcup_{\alpha\in \text{roots}}\alpha^{\perp}=\bigsqcup_j \fC_j,
\end{align*}
where $\fC_j$ are chambers.

For an algebraic variety $M$ with an $\sA$-action, and a cocharacter $\sigma$, let $S$ be a subset of $M^\sigma$, the \emph{attracting set} is 
$$
\Attr_{\sigma} (S) := \{ x\in M \mid \lim_{t\to 0} \sigma (t) \cdot x \in S\}.
$$
Let $\fC$ be a chamber as above, we define
\begin{align*}
    \Attr_{\fC} (S) :=\Attr_{\xi} (S) 
\end{align*}
for a subset $S\subseteq  X^\sA$ and $\xi\in \fC$. Note that the definition does not depend on the choice of $\xi$. 

Let $F \in \Fix_\sA(X)$ be a connected component of the $\sA$-fixed locus, then 
$\Attr_\fC (F)$ is a locally closed subscheme in $X$, and the attraction map 
$$\Attr_\fC (F)\to F, \quad x\mapsto \lim_{t\to 0}\xi(t)\cdot x, $$ 
is an \textit{affine fibration} by the result of Bialynicki--Birula \cite{BB}.

The normal bundle of $F$ in $X$ decomposes into $\sA$-eigen sub-bundles:
\begin{align*}
    N_{F/X}=\bigoplus_{\alpha} N_{F/X}^{\alpha},
\end{align*}
such that $\sA$ acts on $N_{F/X}^{\alpha}$ with weight $\alpha$. Define
\begin{align*}
    N_{F/X}^{+}=\bigoplus_{\alpha(\xi)>0}N_{F/X}^{\alpha}\,,\quad N_{F/X}^{-}=\bigoplus_{\alpha(\xi)<0}N_{F/X}^{\alpha}\,,
\end{align*}
for some $\xi\in \fC$ (equivalently, for all $\xi\in \fC$).

Consider a partial order on $\Fix_\sA(X)$ which is the transitive closure of the following relation:
\begin{align}\label{flow order}
    F_i \preceq F_j \quad \text{if} \quad F_j \cap  \overline{\Attr_\fC (F_i)}\neq \emptyset.
\end{align}
The \emph{full attracting set} is defined as
$$
\Attr_\fC^f (F) := \bigcup_{F \preceq F'} \Attr_\fC (F'). 
$$
We denote by $\Attr_\fC^f$ the smallest $\sA$-invariant closed subset of $X\times X^\sA$ such that $\Attr_\fC^f$ contains the diagonal $\Delta\subseteq  X^\sA\times X^\sA$ and 
\begin{align*}
    (x',y)\in \Attr_\fC^f\text{ and }\lim_{t\to 0} \sigma (t) \cdot x=x' \text{ implies }(x,y)\in \Attr_\fC^f.
\end{align*}
By definition, $\Attr_\fC^f$ is a subset of $\bigcup_{F \in \Fix_\sA(X)} \Attr_\fC^f (F) \times F \subseteq  X \times X^\sA$.
Note also that $\Attr^f_\fC\cap\, (F\times F)=\Delta_F$.

\begin{Definition}\label{stab corr_coh}
Fix a chamber $\fC$ as above. A \textit{stable envelope correspondence} is a $\sT$-equivariant Borel-Moore homology class  supported on $\Attr^f_\fC$: 
$$[\Stab_{\fC}]\in H^\sT(X\times X^\sA)_{\Attr^f_\fC}$$ which satisfies the following two axioms:
\begin{enumerate}[(i)]
\setlength{\parskip}{1ex}

\item For any fixed component $F \in \Fix_\sA(X)$, $[\Stab_{\fC }]\big|_{F\times F}  = e^\sT (N_{F / X}^-) \cdot [\Delta_F]$ for diagonal $\Delta_F\subseteq  F\times F$;

\item For any $F'\neq F$, the inequality $\deg_\sA [\Stab_{\fC }] \big|_{F'\times F} < \deg_\sA e^\sT (N_{F' / X}^-)$ holds.

\end{enumerate}
\end{Definition}
\begin{Remark}\label{rmk on stab corr}
(i)~It follows from the definition that stable envelope correspondence is unique if it exists \cite[Prop.~3.21]{COZZ1}.
(ii)~As noted in \cite[Prop.~3.31]{COZZ1}, for any $\sT$-invariant regular function $\sw\colon X\to \C$, by applying the canonical map \cite[Eqn.~(2.8)]{COZZ1}:
\begin{equation}\label{th can map}\can\colon H^\sT(X\times X^\sA)_{\Attr^f_\fC}\to H^\sT(X\times X^\sA,\sw\boxminus \sw)_{\Attr^f_\fC}\end{equation}
to the stable envelope correspondence $[\Stab_{\fC }]$ and using convolutions, we get the critical stable envelope 
$$\Stab_{\fC }\colon H^\sT(X^\sA,\sw)\to  H^\sT(X,\sw). $$
\end{Remark}
\subsection{Existence of stable envelope correspondence}

\begin{Theorem}\label{thm stab corr}
Let $\overline{\Attr}_\fC(\Delta_F)$ be the closure of the attracting set of the diagonal $\Delta_F\subset F\times F$ in $X\times F$. Then 
$\sum_{F\in \Fix_\sA(X)}[\overline{\Attr}_\fC(\Delta_F)]$ is a stable envelope correspondence.
\end{Theorem}

\begin{proof}
It suffices to show that for every $F\in \Fix_\sA(X)$, $[\overline{\Attr}_\fC(\Delta_F)]$ satisfies axioms (i) and (ii) in Definition \ref{stab corr_coh}. The axiom (i) is obvious. For axiom (ii), the class $(F'\times F\hookrightarrow X\times F)^*[\overline{\Attr}_\fC(\Delta_F)]$ is supported on the subvariety $\overline{\Attr}_\fC(\Delta_F)\cap (F'\times F)$ whose dimension is strictly smaller than $\frac{1}{2}(\dim F+\dim F')$ by Lemma \ref{lem dim bound} below. Therefore
\begin{align*}
    \deg_\sA(F'\times F\hookrightarrow X\times F)^*[\overline{\Attr}_\fC(\Delta_F)]&\leqslant \dim \overline{\Attr}_\fC(\Delta_F)\cap (F'\times F)-\left(\dim \overline{\Attr}_\fC(\Delta_F)+\dim (F'\times F)-\dim (X\times F)\right)\\
    &<\frac{1}{2}(\dim F+\dim F')-\left(\dim \overline{\Attr}_\fC(\Delta_F)+\dim (F'\times F)-\dim (X\times F)\right)\\
    &=\rk N^-_{F'/X}=\deg_\sA e^\sT (N_{F' / X}^-).
\end{align*}
This verifies the axiom (ii).
\end{proof}

In the proof of Theorem \ref{thm stab corr}, we have used the following.

\begin{Lemma}\label{lem dim bound}
Let $F,F'\in \Fix_\sA(X)$ be two \textit{distinct} fixed components. Then 
\begin{align*}
\dim \overline{\Attr}_\fC(\Delta_F)\cap (F'\times F)<\frac{1}{2}(\dim F+\dim F').
\end{align*}
\end{Lemma}

\begin{proof}
Let $\widetilde{F}',\widetilde{F}\in \Fix_\sA(\widetilde{X})$ be connected components of $\widetilde{X}^\sA$ such that $\pi_X|_{\widetilde{F}}\colon \widetilde{F}\to F$ and $\pi_X|_{\widetilde{F}'}\colon \widetilde{F}'\to F'$ are proper, surjective, and generically finite, their existence are shown in the proof of Proposition \ref{tau at fixed loci is flat}. Then we have an induced proper and surjective map 
\begin{align*}
    \overline{\Attr}_\fC(\Delta_{\widetilde{F}})\cap (\widetilde{F}'\times \widetilde{F})\to \overline{\Attr}_\fC(\Delta_F)\cap (F'\times F)
\end{align*}
where $\overline{\Attr}_\fC(\Delta_{\widetilde{F}})$ is the closure of the attracting set of the diagonal $\Delta_{\widetilde{F}}\subset \widetilde{F}\times \widetilde{F}$ in $\widetilde{X}\times \widetilde{F}$. In particular, we have
\begin{align*}
\dim \overline{\Attr}_\fC(\Delta_F)\cap (F'\times F)\leqslant \dim \overline{\Attr}_\fC(\Delta_{\widetilde{F}})\cap (\widetilde{F}'\times \widetilde{F}).
\end{align*}
Let $Z$ be an irreducible component of $\overline{\Attr}_\fC(\Delta_{\widetilde{F}})\cap (\widetilde{F}'\times \widetilde{F})$. For a morphism $M\to \cH$, define $M^\circ=M\times_{\cH}\cH^\circ$ and $M_h=M\times_{\cH}\{h\}$ for $h\in \cH$. Since $\widetilde{X}^\circ$ is an affine variety (see Example \ref{extended nak quiv}), ${\Attr}_\fC(\Delta_{\widetilde{F}^\circ})$ is closed in $\widetilde{X}^\circ\times \widetilde{F}^\circ$. It follows that $\overline{\Attr}_\fC(\Delta_{\widetilde{F}})\cap (\widetilde{F}'^\circ\times \widetilde{F}^\circ)=\emptyset$, in particular, $Z^\circ=\emptyset$. 

We claim that $\overline{\Attr}_\fC(\Delta_{\widetilde{F}})_0$ is an isotropic subvariety of $\widetilde{X}_0\times \widetilde{F}_0$, where the latter is endowed with the symplectic form $\Omega'=\Omega\boxminus \Omega|_{\widetilde{F}_0}$ ($\Omega$ is the symplectic form on the Nakajima quiver variety $\widetilde{X}_0=\cN_{\theta^\delta}(Q^\aux,\bv^\aux,\bd^\aux)$). Consider the two-form $\omega\in \Omega^2(\widetilde{X})$ constricted in Example \ref{extended nak quiv}, it has the property that $\omega|_{\widetilde{X}_0}=\Omega$. Define $$\omega'=\omega\boxminus \omega|_{\widetilde{F}}\in \Omega^2(\widetilde{X}\times \widetilde{F}). $$ 
It is elementary to see that the restriction of $\omega'$ to ${\Attr}_{\fC}(\Delta_{\widetilde{F}})$ vanishes. So the restriction of $\omega'$ to smooth locus of $\overline{\Attr}_\fC(\Delta_{\widetilde{F}})$ vanishes by continuity. Let $W$ be an irreducible component of $\overline{\Attr}_\fC(\Delta_{\widetilde{F}})_0={\Attr}_{\fC}(\Delta_{\widetilde{F}})\cap (\widetilde{X}_0\times \widetilde{F}_0)$. For a general point $w\in W$, there exists a sequence of points $x_1,x_2,\cdots$ in the smooth locus of $\overline{\Attr}_\fC(\Delta_{\widetilde{F}})$ approaching $w$ such that limit of $T_{x_i}\overline{\Attr}_\fC(\Delta_{\widetilde{F}})$ exists as $i\to \infty$ and
contains the tangent space $T_w W$ \footnote{This can be seen by choosing a Whitney stratification of $\overline{\Attr}_\fC(\Delta_{\widetilde{F}})$ for which $\overline{\Attr}_\fC(\Delta_{\widetilde{F}})_0$ is a union of strata.}. Since the restriction of $\omega'$ to $T_{x_i}\overline{\Attr}_\fC(\Delta_{\widetilde{F}})$ vanishes, our claim follows.

Then $\overline{\Attr}_\fC(\Delta_{\widetilde{F}})\cap (\widetilde{F}'_0\times \widetilde{F}_0)$ is an isotropic subvariety of $\widetilde{F}'_0\times \widetilde{F}_0$ by \cite[Lem.~3.4.1]{MO}. And therefore, 
\begin{align*}
\dim Z\leqslant\dim \cH-1+\sup_{h\in \cH}\dim Z_h&\leqslant \dim \cH-1+\dim Z_0\\
&\leqslant\dim \cH-1+\dim \overline{\Attr}_\fC(\Delta_{\widetilde{F}})\cap (\widetilde{F}'_0\times \widetilde{F}_0)\\
&\leqslant\dim \cH-1+\frac{1}{2}(\dim \widetilde{F}_0+\dim \widetilde{F}'_0)\\
{\text{\scriptsize (by the smoothness of $\widetilde{F}'\to \cH$ and $\widetilde{F}\to \cH$)}}\;&=\frac{1}{2}(\dim \widetilde{F}+\dim \widetilde{F}')-1\\
&=\frac{1}{2}(\dim F+\dim F')-1\qedhere  
\end{align*}
\end{proof}

\section{Sheaf theoretic analysis of critical stable envelopes}
A connection between stable envelopes for symplectic resolutions and hyperbolic restrictions is studied by Nakajima \cite[\S 5]{Nak1}. In this section, we extend it to a connection between  critical stable envelopes of symmetric quiver varieties with potentials and hyperbolic restrictions. As an application, we give another proof of the \textit{triangle lemma} \cite[Thm.~4.16]{COZZ1} for symmetric quiver varieties using the associativity of hyperbolic restrictions. 

\subsection{A Steinberg type variety and its homology}

Consider  the Jordan-H\"older morphism in \S \ref{subsec smallness}: 
$$\mathsf{JH}\colon X=\cM_\theta(Q,\bv,\bd)\to \cM_0(Q,\bv,\bd)=Y. $$ 
Let $\fC\subset \Lie(\sA)_\bR$ be a chamber, $\sigma$ be a generic cocharacter in $\fC$, and 
$$\cA_{Y}=\{y\in Y\colon \lim_{t\to 0}\sigma(t)\cdot y\text{ exists}\}$$ be the attracting subvariety of $Y$, with the attraction map and the closed immersion:
\begin{equation}\label{incl and att map}\mathsf a\colon \cA_Y\to Y^\sA, \quad i\colon \cA_{Y}\hookrightarrow Y. \end{equation}
Note that $\cA_{Y}$ is a closed subvariety of $Y$ since $Y$ is affine.  

Let $\D^b_c(-)$ denote the bounded derived category of constructible sheaves and consider  the \textit{hyperbolic restriction} functor \cite{Br}: 
\begin{equation}\label{hy loc fun}\Res^Y_{Y^\sA}=\mathsf{a}_*i^!\colon \D^b_c(Y)\to \D^b_c(Y^\sA). \end{equation}
Define the closed subvariety
\begin{align}
\cA_{X}:=X\times_{Y}\cA_{Y}\subset X.
\end{align}
Note that $\cA_{X}$ is stratified by attracting subvarieties of fixed components on $X$, in particular,
\begin{align}\label{strata of A_X}
\cA_{X}\xlongequal{\text{as set}}\bigcup_{F\in \Fix_\sA(X)}\Attr_\fC(F).
\end{align}
By composing the projection $\cA_{X}\to \cA_{Y}$ and attraction map \eqref{incl and att map}, we get a map $\cA_{X}\to Y^\sA$.
Using this, we define 
\begin{align}
\mathsf Z_{\cA}:=\cA_{X}\times_{Y^\sA} X^{\sA}\subset X \times X^{\sA}.
\end{align}
Let $\C_M$ denote the constant sheaf of a variety $M$. If $M$ is smooth and connected, let $$\cC_M:=\C_M[\dim M]$$ be the shifted constant perverse sheaf; for smooth and disconnected $M$ we define $\cC_M$ to be component-wise the shifted constant perverse sheaf.
\begin{Proposition}[cf. {\cite[Lem.~4]{Nak2}}]\label{prop homology of Steinberg like}
We have a natural isomorphism
\begin{align}\label{homology of Steinberg like}
H_{*}(\mathsf Z_{\cA})\cong \Ext^*_{\D^b_c(Y^\sA)}(\mathsf{JH}^\sA_*\,\cC_{X^\sA}, \Res^Y_{Y^\sA}\mathsf{JH}_*\,\cC_X),
\end{align}
where $H_{*}(\mathsf Z_{\cA})$ is the Borel-Moore homology of $\mathsf Z_{\cA}$, $\mathsf{JH}^\sA=\mathsf{JH}|_{X^\sA}$,
$\Res^Y_{Y^\sA}$ is the hyperbolic restriction functor \eqref{hy loc fun}.
\end{Proposition}

For a connected component $F\subset X^\sA$, \eqref{homology of Steinberg like} reads
\begin{align*}
H_{\dim X+\dim F-i}(\mathsf Z_{\cA}\cap (X\times F))\cong \Ext^i_{\D^b_c(Y^\sA)}(\mathsf{JH}^\sA_*\,\cC_{F}, \Res^Y_{Y^\sA}\mathsf{JH}_*\,\cC_X).
\end{align*}
In particular, we have
\begin{align}\label{top homology of Steinberg like}
H_{\dim X+\dim F}(\mathsf Z_{\cA}\cap (X\times F))\cong \Hom_{\D^b_c(Y^\sA)}(\mathsf{JH}^\sA_*\,\cC_{F}, \Res^Y_{Y^\sA}\mathsf{JH}_*\,\cC_X).
\end{align}

\begin{Lemma}\label{lem top homology}
$\mathsf Z_{\cA}\cap (X\times F)$ contains an irreducible component $\overline{\Attr}_\fC(\Delta_F)$. Moreover,
\begin{align*}
\dim \mathsf Z_{\cA}\cap (X\times F)=\dim \overline{\Attr}_\fC(\Delta_F)=\frac{1}{2}\left(\dim X+\dim F\right),
\end{align*}
and $\overline{\Attr}_\fC(\Delta_F)$ is the unique top dimensional component of $\mathsf Z_{\cA}\cap (X\times F)$. In particular, we have
\begin{align}\label{top homology}
H_{\dim X+\dim F}(\mathsf Z_{\cA}\cap (X\times F))=\bC\cdot[\overline{\Attr}_\fC(\Delta_F)]\,.
\end{align}
\end{Lemma}

\begin{proof}
Obviously $\mathsf Z_{\cA}\cap (X\times F)$ contains $\overline{\Attr}_\fC(\Delta_F)$, and the dimension of the latter is $\frac{1}{2}\left(\dim X+\dim F\right)$. Let $\mathsf Z_{\cA}^\circ=\mathsf Z_{\cA}\times_{\cH/\Gamma} \cH^\circ/\Gamma$, which is open in $\mathsf Z_{\cA}$. Since $\mathsf{JH}\colon X^\circ\to Y^\circ$ is an isomorphism, we have
\begin{align*}
\mathsf Z_{\cA}^\circ\cap (X\times F)=\Attr_\fC(\Delta_{F^\circ}).
\end{align*}
In particular, $\overline{\Attr}_\fC(\Delta_F)$ contains an open subset of $\mathsf Z_{\cA}\cap (X\times F)$, so $\overline{\Attr}_\fC(\Delta_F)$ is an irreducible component of $\mathsf Z_{\cA}\cap (X\times F)$.

Let $W$ be an irreducible component of $\mathsf Z_{\cA}\cap (X\times F)$ which is different from $\overline{\Attr}_\fC(\Delta_F)$. Let $W^\circ=W\times_{\cH/\Gamma} \cH^\circ/\Gamma$ and $W_b=W\times_{\cH/\Gamma}\{b\}$ for $b\in \cH/\Gamma$. Then the same argument in the proof of Proposition \ref{prop smallness} shows that $W^\circ$ is empty and $W_0$ is nonempty, so we have
\begin{align*}
\dim W\leqslant \dim \cH/\Gamma-1+\sup_{b\in \cH/\Gamma}\dim W_b\leqslant \dim \cH/\Gamma-1+\dim W_0\,.
\end{align*}
Let ${\mathsf Z}_{\cA,0}={\mathsf Z}_{\cA}\times_{\cH/\Gamma}\{0\}$. Define $\cA_{\widetilde{X}_0}=\widetilde{X}_0\times_{Y}\cA_{Y}\subset \widetilde{X}_0$, and $\widetilde{\mathsf Z}_{\cA,0}=\cA_{\widetilde{X}_0}\times_{Y^\sA_0} \widetilde{X}_0^{\sA}$. Then there is a natural proper and surjective morphism $\widetilde{\mathsf Z}_{\cA,0}\to {\mathsf Z}_{\cA,0}$. Let $\widetilde{F}\in \Fix_\sA(\widetilde{X})$ be a fixed component that dominates $F$ in the proof of Proposition \ref{tau at fixed loci is flat}. Then $\widetilde{\mathsf Z}_{\cA,0}\cap (\widetilde{X}_0\times \widetilde{F}_0)$ is a Lagrangian subvariety in $\widetilde{X}_0\times \widetilde{F}_0$ by \cite[Prop.~4.5.2]{Nak1}. In particular, $\dim\widetilde{\mathsf Z}_{\cA,0}\cap (\widetilde{X}_0\times \widetilde{F}_0)=\frac{1}{2}\dim \widetilde{X}_0\times \widetilde{F}_0$. It follows that
\begin{align}
\nonumber
\dim W\leqslant \dim \cH/\Gamma-1+\dim {\mathsf Z}_{\cA,0}\cap (X_0\times F_0)&\leqslant \dim \cH/\Gamma-1+\dim\widetilde{\mathsf Z}_{\cA,0}\cap (\widetilde{X}_0\times \widetilde{F}_0)\\ \nonumber 
&=\dim \cH/\Gamma-1+\frac{1}{2}\dim \widetilde{X}_0\times \widetilde{F}_0\\ \nonumber
{\text{\scriptsize (by the smoothness of $\widetilde{X}\to \cH$ and $\widetilde{F}\to \cH$)}}\;&=\frac{1}{2}\left(\dim \widetilde X+\dim \widetilde F\right)-1\\ \nonumber
&=\frac{1}{2}\left(\dim X+\dim F\right)-1 \qedhere
\end{align}
\end{proof}

\subsection{Connections with critical stable envelopes}\label{sect on conn to stab}
The construction of $\cA_Y$, $\cA_X$, and $\mathsf Z_{\cA}$ are $\sT$-equivariant, so  
\begin{align*}
H^\sT_{\dim X+\dim F}(\mathsf Z_{\cA}\cap (X\times F))\cong \Hom_{\D^b_{c,\sT}(Y^\sA)}(\mathsf{JH}^\sA_*\,\cC_{F}, \Res^Y_{Y^\sA}\mathsf{JH}_*\,\cC_X),
\end{align*}
where $\D^b_{c,\sT}(Y^\sA)$ denotes the derived category of $\sT$-equivariant constructible complexes. 
By Lemma \ref{lem top homology}, we have
\begin{align*}
H^\sT_{\dim X+\dim F}(\mathsf Z_{\cA}\cap (X\times F))=\bC\cdot[\overline{\Attr}_\fC(\Delta_F)]\,.
\end{align*}
In particular, the fundamental class $[\overline{\Attr}_\fC(\Delta_F)]$ gives a canonical map
\begin{align}\label{attr induce map}
\sS_F\in \Hom_{\D^b_{c,\sT}(Y^\sA)}(\mathsf{JH}^\sA_*\,\cC_{F}, \Res^Y_{Y^\sA}\mathsf{JH}_*\,\cC_X).
\end{align}
Now let $\sw\colon X\to \bA^1$ be a $\sT$-invariant regular function. Then it restricts to a $\sT$-invariant regular function on $X^\sA$, and descends to $\sT$-invariant regular functions on $Y$ and $Y^\sA$. We will abuse the notation and let $\sw$ denote all of them. Let the vanishing cycle functor associated to $\sw$ be $\varphi_\sw$. Since vanishing cycle functor commutes with proper pushforward and hyperbolic restriction, e.g.~\cite[Prop.~5.4.1]{Nak1}, the map \eqref{attr induce map} induces a map
\begin{align}\label{attr induce map_crit}
\varphi_\sw(\sS_F)\colon \mathsf{JH}^\sA_*\,\varphi_\sw\,\cC_{F}\cong \varphi_\sw\,\mathsf{JH}^\sA_*\,\cC_{F}\xrightarrow{\varphi_\sw(\sS_F)}\varphi_\sw\Res^Y_{Y^\sA}\mathsf{JH}_*\,\cC_X\cong \Res^Y_{Y^\sA}\mathsf{JH}_*\,\varphi_\sw\,\cC_X.
\end{align}
By taking $\sT$-equivariant hypercohomologies on two sides of \eqref{attr induce map_crit}, we obtain
\begin{gather*}
H_\sT(F,\varphi_\sw\,\cC_{F})\cong H_\sT(Y^\sA,\mathsf{JH}^\sA_*\varphi_\sw\,\cC_{F})\xrightarrow{H_\sT(Y^\sA,-)\circ \varphi_\sw(\sS_F)} H_\sT(Y^\sA,\Res^Y_{Y^\sA}\mathsf{JH}_*\,\varphi_\sw\,\cC_X)\cong H_\sT(\cA_Y,i^!\mathsf{JH}_*\,\varphi_\sw\,\cC_X)\\
\xrightarrow[i_*=i_!]{i_!i^!\to \id}H_\sT(Y,\mathsf{JH}_*\,\varphi_\sw\,\cC_X)\cong H_\sT(X,\varphi_\sw\,\cC_{X}),
\end{gather*}
which is the same as the critical convolution map induced by $\can([\overline{\Attr}_\fC(\Delta_F)])\in H^{\sT}(X\times F,\sw\boxminus\sw)_{\mathsf Z_\cA\cap (X\times F)}$, where $\can$ is the canonical map \eqref{th can map}. 
To summarize, we get the critical stable envelope $$\Stab_{\fC}\colon H^\sT(F,\sw)\to H^\sT(X,\sw). \footnote{By definition \cite[(2.3)]{COZZ1}, $H^\sT_i(X,\sw)=H^{-i}_\sT(X,\varphi_\sw\,\omega_X)$, where $\omega_X=\cC_X[\dim X]$ is the dualizing complex on $X$. Then the degree shift in the stable envelope map $H^\sT_i(F,\sw)\to H^\sT_{i+\dim X-\dim F}(X,\sw)$ is exactly accounted by the degree shifts in $\omega_X=\cC_X[\dim X]$ and $\omega_F=\cC_F[\dim F]$.} $$



\subsection{Perverse sheaves on affine symmetric quiver varieties}\label{sect on per on sqv}

Let $\IC_Y$ denote the intersection cohomology complex associated with the trivial rank one local system on a Zariski open dense subset of $Y=\cM_0(Q,\bv,\bd)$ (which can be chosen to be $Y^\circ$). Since $\mathsf{JH}$ is proper birational and small (Theorem \ref{main thm}), there is a natural isomorphism: 
$$\mathsf{JH}_*\,\cC_X\cong \IC_Y. $$ 
Similarly, $\mathsf{JH}^\sA$ is proper birational and small by Proposition \ref{prop smallness fixed locus}, so 
$$\mathsf{JH}^\sA_*\,\cC_{X^\sA}\cong \IC_{Y^\sA}. $$ 
Then, \eqref{top homology of Steinberg like} can be rewritten as
\begin{align*}
H_{\mathrm{top}}(\mathsf Z_{\cA})\cong \Hom_{\D_c^b(Y^\sA)}(\IC_{Y^\sA}, \Res^Y_{Y^\sA}\IC_Y),
\end{align*}
where $H_{\mathrm{top}}(\mathsf Z_{\cA})=\bigoplus_{F\in \Fix_\sA(X)}H_{\dim X+\dim F}(\mathsf Z_{\cA}\cap (X\times F))$ is the component-wise top degree homology.

For a smooth $\sA$-variety $V$ with an $\sA$-equivariant isomorphism $T_V\cong T^*_V$ between tangent and cotangent bundles, we have $\dim \Attr_{\fC}(F)=\frac{1}{2}\left(\dim V+\dim F\right)$ for any fixed component $F$. It follows that there is a natural isomorphism
\begin{align}\label{hyp res isom}
   \Res^V_F\cC_V=(\Attr_{\fC}(F)\to F)_*(\Attr_{\fC}(F)\to V)^!\cC_V\cong (\Attr_{\fC}(F)\to F)_*\bC_{\Attr_{\fC}(F)}[\dim F]\cong \cC_F.
\end{align}
Applying the above discussion to $V=Y^\circ$, and we get a natural isomorphism
\begin{align}\label{isom on open part}
\IC_{Y^\sA}\big|_{Y^{\circ\sA}}\cong \left(\Res^Y_{Y^\sA}\IC_Y\right)\big|_{Y^{\circ\sA}}.
\end{align}

\begin{Proposition}\label{prop sheaf char of stab}
There is an isomorphism
\begin{align*}
    \IC_{Y^\sA}\cong \Res^Y_{Y^\sA}\IC_Y.
\end{align*}
Moreover, $\mathsf S=\sum_{F\in \Fix_\sA(X)}\mathsf S_F\colon \IC_{Y^\sA}\to\Res^Y_{Y^\sA}\IC_Y$ \eqref{attr induce map} is the unique isomorphism that extends the isomorphism \eqref{isom on open part} from $Y^{\circ\sA}$ to $Y^\sA$.
\end{Proposition}

\begin{proof}
Let $j\colon \cA_X\hookrightarrow X$ be the closed immersion, then we have $i^!\mathsf{JH}_*=\mathsf{JH}_*\,j^!$ by the proper base change, where the second $\mathsf{JH}$ is from $\cA_X$ to $\cA_Y$. Using the stratification \eqref{strata of A_X}, we define
\begin{align*}
K_{\succeq I}:=\mathsf{JH}_*\,j_{\succeq I*}j_{\succeq I}^!\cC_X,
\end{align*}
where $I\subset \Fix_\sA(X)$ is a collection of fixed components, and
\begin{align*}
j_{\succeq I}\colon \bigcup_{\substack{F'\in \Fix_\sA(X)\\\exists F\in I,F'\succeq F}}\Attr_\fC(F')\hookrightarrow X
\end{align*}
is the closed immersion of a union of attracting subvarieties. Then $K_{\succeq \emptyset}=0$ and $K_{\succeq \Fix_\sA(X)}=\mathsf{JH}_*\,j^!\cC_X$, and there is a natural map $K_{\succeq I}\to K_{\succeq I'}$ for a pair $I\subset I'$. Suppose $I\subset I'$ are two subsets of $\Fix_\sA(X)$ such that $I'\setminus I=\{F\}$, then we have an exact triangle
\begin{align*}
K_{\succeq I}\to K_{\succeq I'}\to \mathsf{JH}_*k_{F*}k_F^!\cC_X\to K_{\succeq I}[1],
\end{align*}
where $k_F\colon \Attr_\fC(F)\hookrightarrow X$ is the immersion of attracting subvariety into $X$. Applying pushforward along attraction map $\mathsf a\colon \cA_Y\to Y^\sA$ to the above triangle, we get an exact triangle
\begin{align*}
\mathsf a_*K_{\succeq I}\to \mathsf a_*K_{\succeq I'}\to \mathsf{JH}^\sA_*\cC_F\to \mathsf a_*K_{\succeq I}[1],
\end{align*}
where we have used $$\mathsf a_*\mathsf{JH}_*k_{F*}k_F^!\cong \mathsf{JH}^\sA_*(\Attr_\fC(F)\to F)_*k_F^!=\mathsf{JH}^\sA_*\Res^X_F,$$ and $\Res^X_F\cC_X\cong \cC_F$ (which is \eqref{hyp res isom} for $V=X$). Note that $\mathsf{JH}^\sA_*\cC_F\cong \IC_{\mathsf{JH}^\sA(F)}$. By induction on $I\subset\Fix_\sA(X)$, we see that $\mathsf a_*K_{\succeq I}\in \Perv(Y^\sA)$ for all $I$, in particular, $\Res^Y_{Y^\sA}\IC_Y\cong \mathsf a_*K_{\succeq \Fix_\sA(X)}$ is a perverse sheaf. Moreover, $\Res^Y_{Y^\sA}\IC_Y$ has a filtration by $\mathsf a_*K_{\succeq I}$ such that the associated graded, denote $\mathrm{gr}\Res^Y_{Y^\sA}\IC_Y$, is isomorphic to $\bigoplus_{F\in \Fix_\sA(X)}\IC_{\mathsf{JH}^\sA(F)}=\IC_{Y^\sA}$. If we start with setting $\IC_Y$ to be of weight zero, then $\Res^Y_{Y^\sA}\IC_Y$ is pure of weight zero \cite{Br}, so $\Res^Y_{Y^\sA}\IC_Y$ is isomorphic to direct sum of its simple constituents, that is, $\Res^Y_{Y^\sA}\IC_Y\cong \IC_{Y^\sA}$.

We notice that the inverse of the isomorphism \eqref{hyp res isom} is given by the fundamental class $[\Attr_\fC(\Delta_F)]$. Since the restriction of $\sS_F$ to $F^\circ$ is induced by the fundamental class $[\Attr_\fC(\Delta_{F^\circ})]$, we have $\sS|_{Y^{\circ\sA}}=$ \eqref{isom on open part}. Since the support of every simple constituent of $\IC_{Y^\sA}=\bigoplus_{F\in \Fix_\sA(X)}\IC_{\mathsf{JH}^\sA(F)}$ has nontrivial intersection with $Y^{\circ\sA}$, the restriction map
\begin{align}\label{res of hom}
\Hom_{\D^b_c(Y^\sA)}(\IC_{Y^\sA}, \Res^Y_{Y^\sA}\IC_Y)\to \Hom_{\D^b_c(Y^{\circ\sA})}\left(\IC_{Y^{\sA}}\big|_{Y^{\circ\sA}}, \left(\Res^Y_{Y^\sA}\IC_Y\right)\big|_{Y^{\circ\sA}}\right)
\end{align}
is an isomorphism. Thus $\mathsf S$ is the unique homomorphism that extends the isomorphism \eqref{isom on open part}. In particular, $\mathsf S$ is an isomorphism.
\end{proof}

\begin{Remark}\label{rmk traspose}
Consider the opposite chamber $-\fC$, and denote the resulting isomorphism 
$$\mathsf S^-\colon \mathsf{JH}^\sA_*\cC_{X^\sA}\cong \mathsf a^-_*i^{-!}\mathsf{JH}_*\cC_X, $$ where $i^-\colon \cA^-_Y=\{y\in Y : \lim_{t\to\infty}\sigma(t)\cdot y\text{ exists}\}\hookrightarrow Y$ is the closed immersion of the $-\fC$ attracting set ($\sigma$ is a generic cocharacter in $\fC$), and $\mathsf a^-\colon \cA^-_Y\to Y^\sA$ is the attraction map. 

Let $(\mathsf S^-)^{\mathrm{t}}$ be the transpose of $\mathsf S^-$, that is, applying Verdier dual to both sides: $$(\mathsf S^-)^{\mathrm{t}}\colon \bD \mathsf a^-_*i^{-!}\mathsf{JH}_*\cC_X\cong \bD\mathsf{JH}^\sA_*\cC_{X^\sA}.$$ By the properness of $\mathsf{JH}$ and $\mathsf{JH}^\sA$, and the self-duality of $\cC_X$ and $\cC_{X^\sA}$, $(\mathsf S^-)^{\mathrm{t}}$ induces a natural isomorphism 
$$\mathsf a^-_!i^{-*}\mathsf{JH}_*\cC_X\cong \mathsf{JH}^\sA_*\cC_{X^\sA}. $$ 
Using the natural isomorphism of functors $\mathsf a^-_!i^{-*}\cong \mathsf a_*i^{!}$ \cite{Br}, we have
\begin{align*}
(\mathsf S^-)^{\mathrm{t}}\colon \Res^Y_{Y^\sA}\IC_Y\cong \IC_{Y^\sA}.
\end{align*}
Moreover, it is easy to see that $(\mathsf S^-)^{\mathrm{t}}\big|_{Y^{\circ\sA}}$ is the inverse of $\mathsf S\big|_{Y^{\circ\sA}}$. Then it follows from the isomorphism \eqref{res of hom} that
\begin{align}\label{transpose}
(\mathsf S^-)^{\mathrm{t}}\circ \mathsf S=\id,\quad \mathsf S\circ (\mathsf S^-)^{\mathrm{t}}=\id.
\end{align}
We refer to \cite[Lem.~3.29]{COZZ1} for a correspondence version of \eqref{transpose} in a more general setting.
\end{Remark}

\subsection{Sheaf theoretic triangle lemma}

Let $\fC'$ be a face of $\fC$ and $\sA'\subset \sA$ be the subtorus whose Lie algebra is spanned by $\fC'$. Let $Y^{\sA'}$ be the $\sA'$-fixed locus of $Y$, then we have hyperbolic restrictions from $Y$ to $Y^{\sA'}$ along the $\fC'$-attraction direction $$\Res^Y_{Y^{\sA'}}=(\Attr_{\fC'}(Y^{\sA'})\to Y^{\sA'})_*(\Attr_{\fC'}(Y^{\sA'})\to Y)^!\colon \D^b_c(Y)\to \D^b_c(Y^{\sA'}),$$
and from $Y^{\sA'}$ to $Y^{\sA}$ along the $\fC/\fC'$-attraction direction $$\Res^{Y^{\sA'}}_{Y^\sA}=(\Attr_{\fC/\fC'}(Y^{\sA})\to Y^{\sA})_*(\Attr_{\fC/\fC'}(Y^{\sA})\to Y^{\sA'})^!\colon \D^b_c(Y^{\sA'})\to \D^b_c(Y^{\sA}).$$
We notice that there is a Cartesian diagram\footnote{To prove it, we notice that such Cartesian diagram exists if $Y$ is a linear representation of $\sA$. In general, we can embed $Y$ into a linear representation $V$ of $\sA$ and restricts the diagram for $V$ to $Y$.}:
\begin{equation}\label{Cartesian square of attr}
\xymatrix{
\Attr_{\fC}(Y^\sA) \ar@{^{(}->}[r] \ar[d] & \Attr_{\fC'}(Y^{\sA'}) \ar[d]\\
\Attr_{\fC/\fC'}(Y^\sA) \ar@{^{(}->}[r] & Y^{\sA'}\,,
}
\end{equation}
where horizontal arrows are closed immersions and vertical arrows are attractions along the $\fC'$-direction. Applying base change to \eqref{Cartesian square of attr}, we have
\begin{align*}
&(\Attr_{\fC/\fC'}(Y^{\sA})\to Y^{\sA})_*(\Attr_{\fC/\fC'}(Y^{\sA})\to Y^{\sA'})^!(\Attr_{\fC'}(Y^{\sA'})\to Y^{\sA'})_*(\Attr_{\fC'}(Y^{\sA'})\to Y)^!\\
\cong&(\Attr_{\fC/\fC'}(Y^{\sA})\to Y^{\sA})_*(\Attr_{\fC}(Y^{\sA})\to \Attr_{\fC/\fC'}(Y^{\sA}))_*(\Attr_{\fC}(Y^{\sA})\to \Attr_{\fC'}(Y^{\sA'}))^!(\Attr_{\fC'}(Y^{\sA'})\to Y)^!\\
\cong&(\Attr_{\fC}(Y^{\sA})\to Y^{\sA})_*(\Attr_{\fC}(Y^{\sA})\to Y)^!,
\end{align*}
and we obtain the following.
\begin{Proposition}\label{prop hyp res associative}
We have a natural isomorphism of functors
\begin{align}\label{hyp res associative}
    \Res^Y_{Y^\sA}\cong \Res^{Y^{\sA'}}_{Y^\sA}\Res^Y_{Y^{\sA'}}.
\end{align}
\end{Proposition}
Denote the isomorphism in Proposition \ref{prop sheaf char of stab} for $(\sA,\fC)$ by 
$$\mathsf S_{\fC}\colon \IC_{Y^\sA}\cong\Res^Y_{Y^\sA}\IC_Y. $$ 
Replacing $(\sA,\fC)$ with $(\sA',\fC')$, and with $(\sA/\sA',\fC/\fC')$ respectively, we get an isomorphism 
$$\mathsf S_{\fC'}\colon \IC_{Y^{\sA'}}\cong\Res^Y_{Y^{\sA'}}\IC_Y, \quad \mathsf S_{\fC/\fC'}\colon \IC_{Y^{\sA}}\cong\Res{Y^{\sA'}_{Y^\sA}}\IC_{Y^{\sA'}}. $$

\begin{Proposition}\label{prop sheaf tri lem}
We have 
\begin{align}\label{sheaf tri lem}
\Res^{Y^{\sA'}}_{Y^\sA}(\sS_{\fC'})\circ \sS_{\fC/\fC'}=\sS_{\fC}\colon \; \IC_{Y^\sA}\to \Res^{Y^{\sA'}}_{Y^\sA}\Res^Y_{Y^{\sA'}}\IC_Y\underset{\text{\scriptsize\eqref{hyp res associative}}}{\cong}\Res^Y_{Y^\sA}\IC_Y.
\end{align}
\end{Proposition}

\begin{proof}
As we have discussed above, $\sS_{\fC/\fC'}\big|_{Y^{\circ\sA}}$ agrees with the natural isomorphism \eqref{hyp res isom} applied to $V=Y^{\circ\sA'}$ with torus $\sA/\sA'$ action and the chamber $\fC/\fC'$; $\sS_{\fC'}\big|_{Y^{\circ\sA'}}$ agrees with the natural isomorphism \eqref{hyp res isom} applied to $V=Y^{\circ}$ with torus $\sA'$ action and the chamber $\fC'$. It follows from the construction that
\begin{align*}
\text{the composition }\;\cC_{Y^{\circ\sA}}\underset{\text{\scriptsize\eqref{hyp res isom}}}{\cong}\Res^{Y^{\circ\sA'}}_{Y^{\circ\sA}}\cC_{Y^{\circ\sA'}}\underset{\text{\scriptsize\eqref{hyp res isom}}}{\cong}\Res^{Y^{\circ\sA'}}_{Y^{\circ\sA}}\Res^{Y^{\circ}}_{Y^{\circ\sA'}}\cC_Y\underset{\text{\scriptsize\eqref{hyp res associative}}}{\cong} \Res^{Y^{\circ}}_{Y^{\circ\sA}}\cC_Y\;\text{ equals to }\; \cC_{Y^{\circ\sA}}\underset{\text{\scriptsize\eqref{hyp res isom}}}{\cong}\Res^{Y^{\circ}}_{Y^{\circ\sA}}\cC_Y,
\end{align*}
so we have
$\Res^{Y^{\circ\sA'}}_{Y^{\circ\sA}}(\sS_{\fC'}\big|_{Y^{\circ\sA'}})\circ \sS_{\fC/\fC'}\big|_{Y^{\circ\sA}}=\sS_{\fC}\big|_{Y^{\circ\sA}}$. Then the proposition follows from the fact that the restriction map \eqref{res of hom} of the Hom space to $Y^{\circ\sA}$ is an isomorphism.
\end{proof}

Applying the vanishing cycle functor $\varphi_\sw$ to \eqref{sheaf tri lem} followed by taking $\sT$-equivariant hypercohomology, we get the following triangle lemma (cf.\,{\cite[Thm.~4.16]{COZZ1}}).

\begin{Theorem}\label{thm tri lem}
The following diagram 
\begin{equation*}
\xymatrix{
H^\sT(X^\sA,\sw) \ar[rr]^{\Stab_{\fC}} \ar[dr]_{\Stab_{\fC/\fC'}} & & H^\sT(X,\sw), \\
 & H^\sT(X^{\sA'},\sw) \ar[ur]_-{\Stab_{\fC'}} &
}
\end{equation*}
is commutative.
\end{Theorem}

\end{document}